\documentclass[a4paper, 10pt, leqno,final]{article}

\usepackage[T1]{fontenc}		     
\usepackage[utf8]{inputenc}			 
\usepackage[english]{babel}

\usepackage{amsmath}
\usepackage{amssymb}
\usepackage{amsthm}
	\theoremstyle{plain}
		\newtheorem{mainthm}{\textsc{Theorem}}		

		\newtheorem{thm}{Theorem}[section]	
		\newtheorem{cor}[thm]{Corollary}	
		\newtheorem{maincor}{\textsc{Corollary}}		%

		\newtheorem{lem}[thm]{Lemma}		
		\newtheorem{prop}[thm]{Proposition}

	\theoremstyle{definition}
		\newtheorem{defn}[thm]{Definition}	

	\theoremstyle{remark}
		\newtheorem{rem}[thm]{Remark}		
		\newtheorem{note}[thm]{Notation}		

\numberwithin{equation}{section}	

\usepackage{mathtools}		
\usepackage{mathrsfs}		
\usepackage{eucal}			

\usepackage{braket}			
\usepackage[all,pdf]{xy}		
\usepackage{mparhack}		
\usepackage{relsize}			

\usepackage{a4wide}		

\usepackage{booktabs}		
\usepackage{multirow}		
\usepackage{caption}		
\usepackage{rotating}
\usepackage{subfig}

\usepackage{varioref}		
\usepackage{footmisc}

\usepackage{graphicx}		
\usepackage{epsfig}

\usepackage{enumerate}		
\usepackage[colorlinks=true, linkcolor=black, citecolor=black,
urlcolor=blue]{hyperref}		
\mathtoolsset{showonlyrefs}

\newcommand{\cfsa}{\mathcal{CF}^{sa}}
\newcommand{\trasp}[1]{{#1}^\mathsf{T}}	

\newcommand{\MM}{ M}		
\newcommand{\traspm}[1]{{#1}^\mathsf{-T}}		
\newcommand{\iMor}{\mathrm{n_-}}		
\newcommand{\R}{\mathbf{R}}	

\newcommand{\C}{\mathbf{C}}		
\newcommand{\Q}{\mathbf{Q}}		
\newcommand{\circo}{\mathbb{S}}		
\newcommand{\U}{\mathbf{U}}		
\newcommand{\OO}{\mathrm{O}}		

\newcommand{\irel}{I}		
\newcommand{\ispec}{\iota_{\textup{spec}}}		
\newcommand{\ispecomega}{\iota^\omega_{\textup{spec}}}		
\newcommand{\igeo}{\iota_{\textup{geo}}}		
\newcommand{\Sp}{\mathrm{Sp}}

\newcommand{\Lagr}{\Lambda}

\newcommand{\Ddt}{\tfrac{\mathrm{D}}{\mathrm{d}t}}

\newcommand{\Ddtt}{\tfrac{\mathrm{D}^2}{\mathrm{d}t^2}}

\newcommand{\Real}{\mathrm{Re}}

\newcommand{\Lin}{\mathrm{Lin}}
\newcommand{\Graph}{\mathrm{Gr\,}}

\newcommand{\dist}{\mathrm{dist}}

\newcommand{\Gr}{\mathrm{Gr}}

\newcommand{\im}{\mathrm{rge}\,}

\newcommand{\ind}{\mathrm{ind}\,}

\newcommand{\norm}[1]{\left\| #1 \right\|}			
\newcommand{\abs}[1]{\left\lvert #1 \right\rvert}				
\renewcommand{\L}{L}	
\newcommand{\N}{\mathbb{N}}		

\newcommand{\iMas}{\mu_{\scriptscriptstyle{\mathrm{Mas}}}}
\newcommand{\iRel}{\mu_{\scriptscriptstyle{\mathrm{Rel}}}}

\newcommand{\iCLM}{\mu^{\scriptscriptstyle{\mathrm{CLM}}}}

\newcommand{\Z}{\mathbb{Z}}		

\newcommand{\coindex}{\mathrm{n_+}}
\newcommand{\iiindex}{\mathrm{n_-}}
\newcommand{\iindex}[1]{\mu_{\scriptscriptstyle{\mathrm{Mor}}}\left[#1\right]}
\newcommand{\coiindex}[1]{\mathrm{n_+}\left[#1\right]}
\newcommand{\ssgn}[1]{\sgn\left[#1\right]}
\newcommand{\noo}[1]{\overset {\mbox{%
\lower1pt\hbox{${\scriptscriptstyle o}$}}}n^{\mbox{%
\lower2pt\hbox{$\scriptscriptstyle #1$}}}}

\newcommand{\nnull}[1]{\mathrm{n_0}\left[#1\right]}
\newcommand{\nind}[1]{\mathrm{n_-}\left[#1\right]}
\newcommand{\ncind}[1]{\mathrm{n_+}\left[#1\right]}

\DeclareMathOperator{\spfl}{sf}			
\DeclareMathOperator{\sgn}{sgn}		
\DeclareMathOperator{\rk}{rank}		

\renewcommand{\leq}{\leqslant}
\renewcommand{\geq}{\geqslant}
\renewcommand{\hat}{\widehat}

\renewcommand{\=}{\coloneqq}			

\newcommand{\email}[1]{\href{mailto:#1}{\textsf{#1}}}

\usepackage{bm}

\newcommand{\Id}{I}
\newcommand{\X}{\mathcal{X}}               

\title{Instability of semi-Riemannian closed geodesics}
\author{Xijun Hu\thanks{The author is partially supported by NSFC (No.11425105 No. 11790271).},
Alessandro Portaluri
\thanks{The
author is partially supported by the project ERC Advanced Grant 2013
No.~339958 ``Complex Patterns for Strongly Interacting Dynamical Systems ---
COMPAT”, by Prin 2015 ``Variational methods, with applications to
problems in mathematical physics and geometry” No.~$\mathrm{2015KB9WPT\_001}$,
by Ricerca locale 2016 Obem$\_$Rilo$\_$16$\_$01.} , Ran Yang }
\date{\today}

\date{\today}
\begin{document}
 \maketitle

\begin{abstract}

A celebrated result due to  Poincaré affirms that a  closed non-degenerate minimizing geodesic $\gamma$ on an oriented Riemannian surface is hyperbolic. Starting from this classical theorem,  our first main result is  a general instability criterion for timelike and spacelike closed semi-Riemannian geodesics on  both oriented and non-oriented manifolds. A key role is played by  the spectral index, a  new topological invariant that we define through  the spectral flow (being the Morse index truly infinite) of a path of selfadjoint Fredholm operators. A major step in the proof of this result is a {\em new\/} spectral flow formula.

Bott's iteration formula, introduced in \cite{Bot56},  relates in a clear  way the Morse index of an iterated closed Riemannian geodesic and the so-called $\omega$-Morse indices.
Our second result is a semi-Riemannian generalization of the famous
Bott-type iteration formula in the case of closed (resp. timelike closed) Riemannian (resp. Lorentzian) geodesics.

Our last result is a strong instability result obtained by controlling the Morse index of the geodesic and of all of its iterations.

\end{abstract}

\tableofcontents

\section{Introduction}\label{sec:intro}

A celebrated  result proved
by Poincaré in the beginning  of the last century
put on evidence the relation intertwining  the {\em (linear and exponential) instability\/} of a
closed geodesic (as a critical point of the geodesics energy functional on the free loop space) of an oriented Riemannian two-dimensional manifold and its  Morse index. The literature on this criterion is quite broad.  We refer the interested reader to \cite{Poi99,  HS10, Bol88, BT10} and references therein.

Loosely speaking,  a closed geodesic $\gamma$
on $\MM$ is termed {\em linearly stable\/}  if the monodromy matrix  associated to
$\gamma$ splits into two-dimensional rotations. Accordingly, it is diagonalizable and  all Floquet multipliers
belong to  the unit circle $\U$  of the complex plane $\C$. Additionally,
if $1$ is not a Floquet multiplier, we term $\gamma$  {\em non-degenerate\/}. Thus, if $\gamma$ is a stable closed geodesic, then all orbits of the geodesic flow
near $\dot \gamma$ in $T_\gamma\MM$ stay near $\dot \gamma$ for all times. (Cf. Figure \ref{fig:geodetiche}).
\begin{figure}[ht]
 \centering
 \includegraphics[scale=0.15]{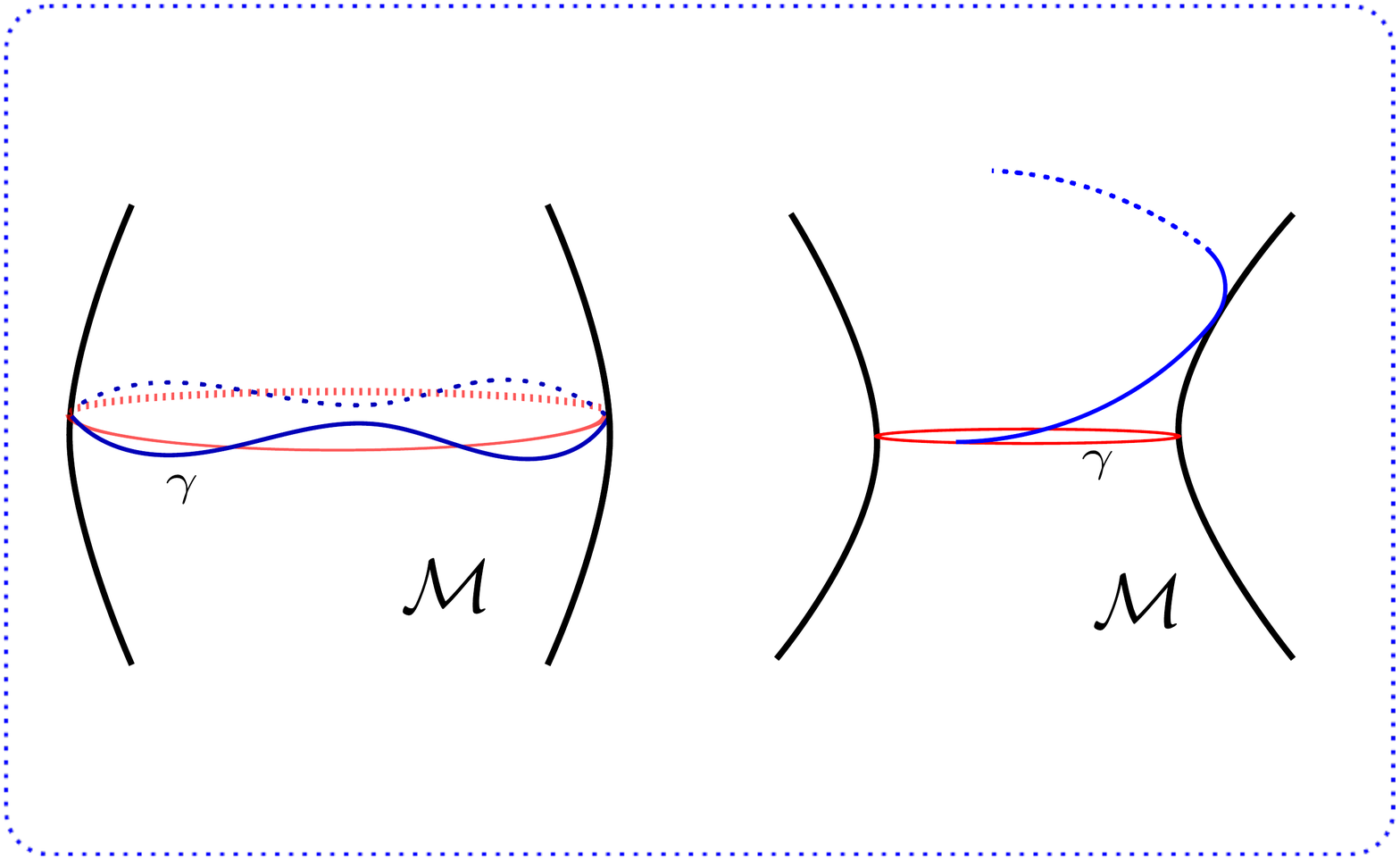}
 \caption{{\color{blue}{Stable}} closed geodesics:{\color{blue}{ nearby
geodesics}} stay in a neighborhood
 of the {\color{red}{closed geodesic $\gamma$}}. Then the
linearized  Poincaré map is diagonalizable and any eigenvalue $\lambda$ is on the unit circle.
 {\color{blue}{Unstable}} closed geodesics: {\color{blue}{nearby geodesics}} have the tendency to diverge. (Image credit to  Hans-Bert Rademacher for the picture available at pag.57 of  \url{http://www.math.uni-leipzig.de/~rademacher/tianjin.pdf}). }\label{fig:geodetiche}
\end{figure}

In  1988 the  Poincaré instability criterion for closed geodesics
was generalised  in several interesting directions by
D. V. Treschev in \cite{Tres88}.  In this paper the author describes  the connection intertwining
the {\em Morse index\/} of a closed non-degenerate Riemannian geodesic
$\gamma$ and the spectrum of the Poincaré map.  More precisely, denoting by $\iMor(\gamma)$
the Morse index of $\gamma$ as a critical
point of the geodesic energy functional on the free loop space of the
$(n+1)$-dimensional Riemannian manifold $\MM$,    the author proved that
if either  $\gamma$ is  a non-degenerate oriented closed geodesic such that
$n+\iMor(\gamma)$ is odd or  $\gamma$ is nonoriented non-degenerate closed
geodesic and $n+\iMor(\gamma)$ is even then $\gamma$ is linearly unstable. Several years later,  the first author and his collaborator in \cite{HS10} proved  a
generalization of the aforementioned  result, dropping the non-degeneracy assumption.  Of the same flavor is a
very recent instability
result proved by the first author and his collaborator in the case of
Hamiltonian systems. (We
refer the interested reader to \cite[Subsection 4.3, pag. 762]{HS09}).
Due to its interest in dynamical systems, a big effort has been given in the
investigation of
stability properties of closed geodesics on Riemannian manifolds as well as of
periodic solutions of more general Lagrangian  systems (cf. \cite{LL02,BJP14,
BJP16, HS09}
and references therein) under the standard Legendre convexity condition.

Dropping the positivity
assumption of the metric tensor is a quite challenging task. The first
problem  is that the critical points of the geodesic energy functional, have in general infinite Morse index and  co-index.  However, in this {\em strongly indefinite\/} situation a natural substitute of the Morse index is represented by
a topological invariant known in literature as the {\em spectral flow\/}.

The spectral flow is naturally associated to a path of  selfadjoint Fredholm operators  arising from the second variation along the geodesic $\gamma$. This invariant was introduced by Atiyah, Patodi and Singer in \cite{APS76} in their study of index theory on manifolds
with boundary and since then many interesting properties and applications have been established. (Cf., for instance, \cite{PPT04, MPP07, MPW17, HP17b, PW16} and references therein).  In general the
spectral flow depends on the homotopy class of the whole path and not only on its ends. However in the special case of geodesics on semi-Riemannian manifolds things are simpler since it depends only on the endpoints of the path and therefore it can be considered a relative form of Morse index known in literature as  {\em relative Morse index.\/}

It is well-known that this invariant is strictly related to a  symplectic
invariant known in literature as
{\em Maslov index\/}; by using these two integers, several  generalization of
the celebrated
Morse Index Theorem are available. (We refer the interested reader to
cf. \cite{APS08, MPP05, MPP07, HS09, GPP03, LZ00a, LZ00b,Por08, RS95} and references
therein). Very recently new spectral flow formulas has been established in the study of heteroclinic and
homoclinic orbits of Hamiltonian systems. (Cfr. \cite{BHPT17, HP17a, HPY17}).
{{\subsection*{Acknowledgements}

  The second named author warmly thanks  all faculties and staff at the Mathematics Department of the Shandong  University (Jinan)  for providing excellent working conditions during his stay. 
  
  We are deeply grateful to the anonymous referees for their careful reading of our manuscript and their insightful comments and suggestions.
}}
\vskip1truecm


\section{Description of the problem and main
results}

The aim of this section is to describe the problem, the main results and to introduce some basic definitions and finally to fix our  notation.

Let $(\MM, \mathfrak g)$ be a $(n+1)$-dimensional semi-Riemannian
manifold, where $\mathfrak g$ is a metric
tensor of  index,  $\iiindex(\mathfrak g)=: \nu \in \{0, \dots, n+1\}$. If $\nu=0$ (resp. $\nu=1$) the
pair $(\MM, \mathfrak g)$  defines a Riemannian (resp. Lorentzian) manifold.
Let  $\nabla$ denote the Levi-Civita connection of $\mathfrak g$ and  let
$\mathcal R$ be the corresponding curvature tensor chosen with the following
sign convention
\begin{equation}
\mathcal R(X,Y)=[\nabla_X, \nabla_Y]- \nabla_{[X,Y]}
\end{equation}
where $[\nabla_X, \nabla_Y]\= \nabla_X \nabla_Y-\nabla_Y \nabla_X$. Denoting  by $\Ddt$ the covariant derivative, we recall that
a {\em closed geodesic\/} on $M$ is a smooth solution $\gamma\colon
[0,T] \rightarrow M$ of  the following  problem
\begin{equation}\label{eq:geo}
\begin{cases}
 \Ddt \dot \gamma(t)=0, \qquad t \in [0,T]\\
 \gamma(0)= \gamma(T)\\
 \dot\gamma(0)=\dot \gamma(T)
\end{cases}
 \end{equation}
 where $\cdot $ denotes the time derivative. It is well-known that, if $\gamma:[0,T]\to M$ is a geodesic, then there exists a
constant $e_\gamma$ such that
\begin{equation}\label{eq:energia-intro}
e_\gamma \= \mathfrak g\big(\dot  \gamma, \dot \gamma).
\end{equation}
The value of $e_\gamma$ given in Equation \eqref{eq:energia-intro} determines the causal
character of the geodesic; more precisely,
$\gamma$ is termed {\em timelike, lightlike\/} or {\em spacelike\/} if
$e_\gamma$ is negative, zero or
positive, respectively. Given a (closed) geodesic $\gamma$, a  {\em Jacobi field\/} is a
smooth  vector field $\xi$
along $\gamma$  that satisfies the second order linear differential equation
\begin{equation}\label{eq:jacobi-intro}
-\Ddtt \xi  (t) +\mathcal R(\dot \gamma (t) , \xi  (t)) \dot \gamma (t) =0, \qquad
t\in [0,T].
\end{equation}
We denote by
$\mathfrak P_\gamma: T_{\gamma(0)}M \oplus T_{\gamma(0)}M \to  T_{\gamma(0)}M
\oplus T_{\gamma(0)}M$ the map
 defined by
$
\mathfrak P_\gamma(v, v')= \left(\xi(T), \Ddt \xi(T)\right)
$
where $\xi$ is the unique Jacobi field along $\gamma$ such that $\xi(0)=v $ and $\Ddt \xi(0)=v'$.
We  assume that $\gamma$ is a closed spacelike (resp. timelike) geodesic, namely for
any $t\in[0,T]$, $\mathfrak g(\dot{\gamma}(t),\dot{\gamma}(t))=\abs{\dot \gamma(t)}^2$
(resp. $\mathfrak g(\dot{\gamma}(t),\dot{\gamma}(t))=-\abs{\dot \gamma(t)}^2$ ) . We let $e_{0}(0)\=\frac{\dot{\gamma}(0)}{|\dot{\gamma}(0)|}$ if $\gamma$ is spacelike
and $e_{n-\nu+1}(0)\=\frac{\dot{\gamma}(0)}{|\dot{\gamma}(0)|}$ if $\gamma$ is timelike. Then we can find $n$ linearly
independent  $\mathfrak g$-orthonormal vectors in $T_{\gamma(0)}M$ such that the following set
 \begin{multline}
 \big\{e_{1}(0),\ldots,e_{n}(0)\big\} \textrm{ in the spacelike case } \\ \big(\textrm{resp. } \big\{e_{0}(0),\ldots,e_{n-\nu}(0), e_{n-\nu+2}(0),\ldots,e_{n}(0)\big\}
\textrm{ in the timelike case}\big)
 \end{multline}
   is a basis of the subspace $\displaystyle
N_{\gamma(0)}M=\Set{v\in T_{\gamma(0)}M\mid \mathfrak g(v,\dot{\gamma}(0))=0}.$
Since parallel transport along $\gamma$ is
a $\mathfrak g$-isometry, then it follows that $\{e_{1}(t),\ldots,e_{n}(t)\}$ is a $\mathfrak g$-orthonormal basis of $N_{\gamma(t)}M$ in the spacelike case whilst  $\{e_{0}(t),\ldots,e_{n-\nu}(t), e_{n-\nu+2}(t), \ldots, e_n(t)\}$ is a $\mathfrak g$-orthonormal basis of $N_{\gamma(t)}M$ in the timelike case.
Let us denote by $\mathfrak P_\gamma^{\perp_\mathfrak g}:
N_{\gamma(0)} \oplus N_{\gamma(0)}\to N_{\gamma(0)} \oplus N_{\gamma(0)}$ the restriction to the
normal subspace at $\gamma(0)$ of  the
linearized {\em Poincaré map\/} defined by $
 \mathfrak P_\gamma^{\perp_\mathfrak g}(v, v')= \left(\xi(T), \Ddt \xi(T)\right)$
where $\xi$ is the unique Jacobi field $\mathfrak g$-orthogonal to $\dot \gamma$
such that
$\xi(0)=v$ and $\Ddt \xi(0)= v'$.
\begin{defn}\label{def:stability-intro}
A closed semi-Riemannian geodesic $\gamma$  is termed
{\em linearly stable\/} if the linearized Poincaré map $\mathfrak P_\gamma^{\perp_\mathfrak g}$
is semisimple and its spectrum $\sigma(\mathfrak P_\gamma^{\perp_\mathfrak g}) $ lies on the unit circle $\U$ of the
complex plane. Otherwise it is termed {\em linearly unstable\/}.
\end{defn}
 If
 $\gamma \in \Lambda\MM$  is a (non-constant) closed geodesic in $(\MM, \mathfrak g)$
we can introduce, by means of the parallel transport,  a  trivialization of the
pull-back bundle $\gamma^{*}(T\MM)$ of $T\MM$ along $\gamma$
by choosing a parallel $\mathfrak g$-orthonormal frame $\mathfrak E$
along $\gamma$ given by $(n+1)$-parallel linearly independent vector fields
$ e_0, \dots  e_n$ and by means of this parallel
trivialization,  the metric tensor $\mathfrak g$ reduces to the constant
indefinite scalar product $g$ in $\R^{n+1}$ of constant index $\nu$.
Writing the Jacobi  vector field along $\gamma$ in local coordinates as $\xi(t)=\sum_{i=0}^{n}u_{i}(t)e_{i}(t)$, inserting the above expression
into the Equation \eqref{eq:jacobi-intro} and by taking the $\mathfrak g$-scalar product with $e_j$, we reduce
it to the linear second order system of ordinary differential equations
\begin{equation}\label{eq:Space-Original-intro}
 -\bar{G}\ddot u (t) + \bar{R}(t)u(t)=0, \qquad t \in [0,T]
\end{equation}
where $\bar R_{ij}\= \mathfrak g\big(\mathcal R(\dot \gamma, e_i) \dot \gamma, e_j\big)$ and $\bar{G}= \begin{bmatrix}I_{n+1-\nu} & 0\\0 &-\Id_\nu \end{bmatrix}$. Being the map $\mathcal R(\dot \gamma, \cdot) \dot\gamma $
$\mathfrak g$-symmetric, it follows that the matrix $\bar R=[\bar R_{ij}]_{i,j=0}^n$ is symmetric; moreover in the spacelike case,  
$ \bar{R}_{0i}(t)=\bar{R}_{i0}(t)=0$ for any $i=0,\dots,n$.
Since the Jacobi field $\xi$ is
$T$-periodic, inserting the local expression of $\xi$ into the equation $\xi(0)=\xi(T)$, we get the following boundary
condition for $u$:
\begin{equation}
u(0)=\bar A\,u(T)   \textrm{ where }  \bar{A}=[a_{ij}]_{i,j=0}^n.
\end{equation}
It is worth to observe that, since $\dot{\gamma}(0)=\dot{\gamma}(T)$, then we have $a_{00}=1$ and $a_{0j}=0$ for any $j=1,\dots,n$. Furthermore, by a direct
computation it follows also  that $\bar A$ is $\mathfrak g$-orthogonal.
By construction, the Morse-Sturm system  given in Equation \eqref{eq:Space-Original-intro} decouples into a scalar differential
equation (corresponding to the Jacobi field along $\dot \gamma$) and a differential system in $\R^n$
(corresponding to the restriction of the Jacobi deviation equation to vector fields $\mathfrak g$-orthogonal
to $\gamma$). Similarly in the timelike case, we have
\begin{multline}
\bar R_{(n-\nu+1)i}(t)=\bar R_{i(n-\nu+1)}(t)=0  \textrm{ and }\\
a_{(n-\nu+1)(n-\nu+1)}=1, \ a_{(n-\nu+1)j}=0 \textrm{ for } j \neq n-\nu+1.
\end{multline}
Depending on the causal character of $\gamma$, we can distinguish two cases:
\begin{itemize}
\item if $\gamma$ is {\em spacelike\/} then we have
\begin{equation}\label{eq:Space-Reduced-space-intro}
-G \ddot u(t)+\widehat{R}(t)u(t)=0, \qquad t \in [0,T]
\end{equation}
where $u(t)=\trasp{(u_{1}(t),\ldots,u_{n}(t))},  G\= \begin{bmatrix}I_{n-\nu} & 0\\0 &-\Id_\nu \end{bmatrix}$
and $\widehat{R}(t)=[ \bar{R}_{ij}(t)]_{i,j=1}^n$. Correspondingly the matrix $\bar A$ reduces to
\[
A=[a_{ij}]_{i,j=1}^n.
\]
\item if $\gamma$ is {\em timelike\/} then we have
\begin{equation}\label{eq:Space-Reduced-time-intro}
-G \ddot  u(t)+\widehat{R}(t)u(t)=0, \qquad t \in [0,T]
\end{equation}
where $u(t)=\trasp{(u_{0}(t),\dots, u_{n-\nu}(t), u_{n-\nu+2}, \dots, u_{n}(t))},$
$G\ =\begin{bmatrix}I_{n+1-\nu} & 0\\0 &-\Id_{\nu-1} \end{bmatrix}$ and $\widehat{R}(t)=[ \bar{R}_{ij}(t)]_{i,j=0}^n$ and
$i,j \neq n-\nu+1$. In this case the matrix $\bar A$ reduces to
\[
A=[a_{ij}]_{i,j}^n \textrm{ where } i,j =0, \dots, n \textrm{ and } i,j \neq n-\nu+1.
\]
\end{itemize}
Now, we are entitled to introduce the following definition.
\begin{defn}\label{def:oriented}
A closed geodesic $\gamma$ is termed oriented if $\det A= 1$; non-oriented if $\det A= -1$.
\end{defn}
\begin{rem}
We observe that the orientation of a geodesic $\gamma$ is independent either on the signature of $\mathfrak g$ or on the causal character of $\gamma$.
\end{rem}
\begin{note}
In short-hand notation, in the rest of the paper we will denote either a spacelike or a  timelike geodesic with
the same symbol, if not explicitly stated.
\end{note}
Given a spacelike (resp. timelike) closed geodesic $\gamma$, let
$t \mapsto \Psi(t)$ be the {\em flow\/} of the Morse-Sturm system
given in Equation \eqref{eq:Space-Reduced-space-intro} (resp. Equation \eqref{eq:Space-Reduced-time-intro});
thus  for every $t \in [0,T]$, $\Psi$ is the unique linear isomorphism of $\R^n \oplus \R^n$ such that $\displaystyle
 \Psi(t)\big(G\dot u(0), u(0)\big)= \big(G\dot u(t), u(t)\big),$
where $u$ is a solution of Equation \eqref{eq:Space-Reduced-space-intro}
(resp. Equation \eqref{eq:Space-Reduced-time-intro}).
We observe that $\Psi$ is a
smooth curve in the general linear
group of $\R^n \oplus \R^n$ satisfying the matrix differential equation $\dot
\Psi (t)= K(t)\Psi(t)$ with
initial condition $\Psi(0)=\Id$ where $K$ is given by $
K(t)\=\begin{bmatrix}
       0 & \hat{R}(t) \\
       G& 0
      \end{bmatrix}.$
The symmetry of $\hat{R}$ implies that $\Psi$ is actually a (smooth) curve in $\Sp(2n,\R)$.
We denote by $J$  the standard complex structure, given by $\displaystyle
J\= \begin{bmatrix}
       0 & -\Id\\
       \Id&0
      \end{bmatrix}$
where $\Id$ denotes the identity in the appropriate dimension.
Observing that $\trasp{A}G A= G$,  it follows  that the operator
$\displaystyle
A_d\=\begin{bmatrix}
\traspm{A} & 0 \\
0 & A
\end{bmatrix}$
lies in $\Sp(2n, \R)$ being, in fact, $\trasp{A_d}JA_d= J$.
By taking into account the $\mathfrak g$-orthonormal periodic trivialization of
$\gamma^*(TM)$, the induced linearized Poincaré map is given by
\begin{equation}\label{eq:poincare-1}
\mathcal P(T)\=   A_d\Psi(T) \in \Sp(2n, \R).
\end{equation}
\begin{rem}
We observe that, in terms of the the operator $A_d$, a spacelike (resp. timelike) geodesic $\gamma$ is {\em linearly stable\/} if the symplectic  matrix $\mathcal P(T)$ is linearly stable.
\end{rem}
In order to introduce both the geometric and analytic indices, we start to embed the second order self-adjoint differential operator coming out from the Jacobi deviation equation, into a one parameter family of operator. Now,   for any $s \in [0,+\infty)$, we   introduce  the closed (unbounded) operator
$
 \mathcal A_s: \mathscr D(\mathcal A_s) \subset L^2([0,T]; \R^n) \to L^2([0,T];
\R^n)
$
having  domain $\mathscr D(\mathcal A_s)\= \{u\in W^{2,2}([0,T],\R^{n}):
u(0)=Au(T),\dot{u}(0)=A\dot{u}(T)\}$
(independent on $s$) and  defined by
\begin{equation}\label{eq:gli-operatori}
 \mathcal A_s(u)(t)\= -G\dfrac{d^2}{dt^2} + \widehat R(t)+s\,G.
\end{equation}
It is well-known (cf., for instance, \cite{GGK90}) that,
for each $s \in [0,s_0]$ the operator $\mathcal A_s$ is
a closed Fredholm operator and selfadjoint in $L^2([0,T]; \R^n)$.
Since the domain $\mathscr D(\mathcal A_s)$ doesn't depend
on $s$, the path $s \mapsto \mathcal A_s$  can be seen as a path of bounded
Fredholm operators from
$\mathcal W\= W^{2,2}([0,T];\R^n) \cap W^{1,2}_A([0,T];\R^n)$ into $L^2([0,T],
\R^n)$ which  are selfadjoint when
regarded as operators on $L^2$, where
$
 W^{1,2}_A([0,T]; \R^n)\= \{u \in W^{1,2} ([0,T], \R^n): u(0)= A\,u(T)\}.
$
For any $c \in [0,1]$,  $s \in [0,+\infty)$ and for any $\omega \in \U$, we define the operator
\begin{equation}\label{eq:Acs}
\mathcal A^\omega_{c,s}=-G \dfrac{d^2}{dt^2} + c \widehat R(t)+ sG, \qquad t \in [0,T]
\end{equation}
on  the Hilbert  space
\begin{equation}\label{spazio-Acs}
E^2_\omega([0,T])\=\Set{u\in W^{2,2}([0,T], \C^n)| u(0)=\omega A u(T), \  \dot u(0)=\omega A \dot u(T)}.
\end{equation}
As we will prove in Corollary \ref{thm:lemma3}, there exists $s_0$ sufficiently large such that for $s \geq s_0$, the operator
$\mathcal A^\omega_{1,s}$ is non-degenerate.  Now we are entitled to define the spectral index of a
closed geodesic $\gamma$.
\begin{defn}\label{def:spectral-index}
Under the previous notation, we term {\em $\omega$-spectral index\/} of the closed non-lightlike geodesic $\gamma$, the integer
$\ispecomega(\gamma)$ defined by
\begin{equation}\label{eq:spectral-index}
 \ispecomega(\gamma)\=  \spfl(\mathcal A^\omega_{1,s}; s \in [0,s_0]).
\end{equation}
\end{defn}
\begin{rem}
As already observed, since  for $s \geq s_0$, the operator
$\mathcal A^\omega_{1,s}$ is non-degenerate, the spectral index given in Definition \ref{def:spectral-index} is well-defined, i.e. it is independent on $s_0$.
\end{rem}
We now introduce the Hamiltonian system
\begin{equation}\label{eq:Dcs}
\dot z(t)= J D_{c,s}(t) z(t), \qquad t \in [0,T]
\end{equation}
for $\displaystyle
D_{c,s}(t)\=
\begin{bmatrix}
G & 0 \\0 & - c\widehat R(t) -sG
\end{bmatrix}$, $s \in [0,+\infty)$
and we denote by $\Psi_{c,s}$ its fundamental solution.
\begin{defn}\label{def:Maslov-geodesic}
Let $\gamma$ be a closed non-lightlike geodesic and let $\mathcal P_\omega:[0,T] \to \Sp(2n)$ be the path pointwise given by $\mathcal P_\omega(t)\= \omega A_d\Psi_{1,0}(t)$.
We define the {\em $\omega$-geometric index of $\gamma$\/}  as follows
 \begin{equation}\label{eq:indice-geometrico}
  \igeo^\omega(\gamma)\=\iota_1(\mathcal P_\omega(t); t\in [0,T])
 \end{equation}
 where $\iota_1$ is the Maslov-type index (cf. Appendix \ref{sec:appendix}  and references therein for the
definition and the main properties of $\iota_1$).
 \end{defn}
 \begin{note}
 In shorthand notation, we will denote
 \begin{itemize}
\item  the path $\mathcal P_1$ (obtained by setting $\omega=1$) by $\mathcal P$;
\item $\igeo^1(\gamma) $ by $\igeo(\gamma)$;
\item $\ispec^1(\gamma)$ by $\ispec(\gamma)$.
\end{itemize}
 \end{note}
Our first main result reads as follows.
\begin{mainthm}\label{thm:Lemma6}({\bf An $\omega$-spectral flow formula\/})
Under the previous notation, the following spectral flow formula holds:
\[
 \ispecomega(\gamma) +\dim\ker ( A -\omega \Id) =  \igeo^\omega(\gamma).
 \]
\end{mainthm}
As a  direct consequence of the $\omega$-spectral flow formula, we immediately get  a new Morse-type Index Theorem.
\begin{maincor}\label{thm:Lemma7}({\bf A Morse Index Theorem \/})
Under the previous notation, we have
\[
\ispec(\gamma)+\dim\ker(A-\Id)= \igeo(\gamma).
\]
\end{maincor}
\begin{proof}
The proof of this result immediately follows by setting $\omega=1$. \end{proof}
We denote by $\gamma^{(m)}:[0,mT]\to M$ the $m$-th iteration of the  geodesic $\gamma$, defined by
\[
\gamma^{(m)}(t)\= \gamma(t-jT), \quad jT \leq t \leq (j+1)T, \ j=0, \dots, m-1.
\]
Another  contribution of the present paper is  represented by a sort of  semi-Riemannian version of the Bott-type iteration formula which plays a crucial role in the instability criteria that we shall prove. It is worth to
note that, with respect to the classical case,  in our framework the  Legendre convexity condition does not hold.
\begin{mainthm}({\bf Semi-Riemannian Bott-type formula\/}) \label{thm:Bott-SR}
Let $(M, \mathfrak g)$ be a semi-Riemannian manifold and $\gamma$ be a closed
non-lightlike closed geodesic. For any $m \in \N$, the
following iteration formula holds
\begin{equation}\label{eq:Bott-equation}
\ispec\big(\gamma^{(m)}\big)=\sum_{\omega^m=1} \ispecomega(\gamma)
\end{equation}
\end{mainthm}
By using the spectral flow formula stated in Theorem \ref{thm:Lemma6} and by taking into account the homotopy  properties of $\Sp_\omega^*(2n)$ (cf. Subsection \ref{subsec:Maslov}), we get our main linear instability result for spacelike and timelike  closed semi-Riemannian geodesics.
\begin{mainthm}\label{thm:main-deg}{\bf(Semi-Riemannian Instability criterion)}
Let $(M, \mathfrak g)$ be a $(n+1)$-dimensional semi-Riemannian manifold
of index $\nu$ and let $\gamma:[0,T] \to M$ be a  closed non-lightlike geodesic. If
\begin{itemize}
\item[]{\bf (OR)} $\gamma$   is oriented and $\ispec(\gamma) +n$   is odd
\item[]{\bf (NOR)} $\gamma$   is nonoriented and $\ispec (\gamma) +n$ is even
\end{itemize}
then $\gamma$ is linearly unstable.
\end{mainthm}
{\begin{rem}
It is worth to observe that the spectral flow techniques were successfully developed for proving  Morse type index theorem  (cfr. \cite[Theorem 5.6]{BP10} which holds even for  closed non-lightlike geodesic) and Bott's iteration formula (cf. \cite[Theorem 5.3]{JP08}) for closed semi-Riemannian geodesics. Please note that here we choose a different operator path to define the spectral flow. The choice made by authors in  the \cite{BP10} is in the direction of extending the classical Morse index theorem for closed Riemannian geodesics whilst our choice   comes from the original purpose of investigating the instability of closed semi-Riemannian geodesics.
\end{rem}

An immediate consequence of Theorem
\ref{thm:main-deg} is the following instability criterion for closed Riemannian (resp. Lorentzian)  geodesics (resp.timelike geodesics).
\begin{maincor}\label{thm:Riemann}
Let $(M, \mathfrak g)$ be a $(n+1)$-dimensional
Riemannian manifold (resp. Lorentzian) and let $\gamma:[0,T] \to M$ be a closed (resp. timelike closed)
geodesic.
If one of the following two alternatives hold
\begin{itemize}
\item[]{\bf (OR)} $\gamma$   is oriented (resp. oriented and timelike) and $
\iMor(\gamma) +n$   is odd
\item[]{\bf (NOR)} $\gamma$   is non-oriented (resp. non-oriented and
timelike) and $
\iMor(\gamma) +n$  is even
\end{itemize}
then the geodesic  is linearly unstable.
\end{maincor}
\begin{proof}
 The proof of this result readily follows by  Theorem
 \ref{thm:main-deg} once observed that for closed Riemannian  (resp. timelike closed Lorentzian) geodesics, $\ispec(\gamma)=\iMor(\gamma)$(cf. \cite[Section $3$]{HS09} or \cite[Section $3$]{HS10}).
\end{proof}
\begin{rem}
It is worth noting that the celebrated Poincare's instability criterion, can be recovered by  Corollary \ref{thm:Riemann}.  In fact,  if $\gamma$ is
a minimizing closed geodesic, then
$\iMor(\gamma)=0$ and on a Riemannian surface $n=1$; thus  $\iMor(\gamma)+n =1$.
\end{rem}
Before describing the last  result, we introduce the notion of strong stability for closed geodesic on Riemannian or Lorentzian
manifolds.
\begin{defn}\label{def:strong stable}
Let $\gamma$ be a  closed (resp. timelike closed) Riemannian (resp. Lorentzian)
geodesic and let $\mathcal P(T)$ be the corresponding monodromy matrix. We say that
$\gamma$ is termed {\em strongly stable\/} if,  there exists
$\varepsilon>0$ such that any symplectic matrix $M$ with
$\norm{ M-\mathcal P(T)}\leq\varepsilon$ is linearly stable.
\end{defn}

\begin{mainthm}\label{cor:non-strong stable}
Let $(M, \mathfrak g)$ be a  $(n+1)$-dimensional Riemannian (resp. Lorentzian) manifold and let $\gamma$ be a  closed (resp. timelike closed) Riemannian (resp. Lorentzian)
geodesic. We assume that
\[
\iMor\big(\gamma^{(m)}\big)=0, \qquad  \textrm{ for any } m\in \N.
\]
Then $\gamma$ is not strongly stable.
\end{mainthm}

\section{Variational preliminaries}\label{sec:geometric-variational-preliminaries}

The aim of this section is to introduce the basic  geometric
and variational setting  and to fix our notations.
Our basic references are \cite{Kli78}.

It is well-known that,  closed semi-Riemannian  geodesics are critical
points of the {\em geodesic energy functional\/} defined
on  the {\em free loop space\/}
$\Lambda \MM$ of $\MM$ where $\Lambda \MM$ denotes the Hilbert
manifold of
all closed curves in $\MM$ of Sobolev class $W^{1,2}$. Let $\circo$ be the circle, viewed as the quotient  $[0,T]/\{0,T\}$,
$\Lambda \MM\= W^{1,2}(\circo,\MM)$ be the infinite dimensional
Hilbert manifold
(cf. \cite[Theorem 1.2.9, pag.13]{Kli78}) of all loops $\gamma: [0,T] \to
\MM$ (namely
$\gamma(0)=\gamma(T)$)
of Sobolev class $W^{1,2}$. For any
$\gamma \in \Lambda \MM$,
the tangent space $T_\gamma \Lambda \MM$ can be identified with the Hilbert
space of all sections of the
pull-back bundle $\gamma^*(T\MM)$ (namely the bundle of all periodic
vector fields along
$\gamma$) of
Sobolev class $W^{1,2}$ (cf. \cite[Theorem 1.3.6, pag.19]{Kli78}) that we'll
denote by $\mathcal H_\gamma$; thus
\[
\mathcal H_\gamma=\{\xi \in W^{1,2}(\circo, T\MM): \tau \circ \xi=
\gamma\},
\]
where $\tau:TM \to M$ denotes the canonical tangent bundle projection.
We consider the {\em geodesics energy functional\/}
$E: \Lambda \MM \to \R$ given by
\begin{equation}\label{eq:funzionale-energia}
E(\gamma)= \dfrac12 \int_0^T\mathfrak g\big(\dot \gamma(t), \dot \gamma(t) \big)\,
dt
\end{equation}
and we observe that $E$ is
differentiable and for any $\xi \in T_\gamma \Lambda \MM$
\begin{equation}
dE (\gamma ) [\xi ] =  \int_0^T \mathfrak g\left( \Ddt \xi(t),
\dot \gamma(t)\right)dt.
\end{equation}
(Cf.  \cite[Lemma 1.3.9, pag.21]{Kli78}, for further details).
By standard regularity arguments and by performing integration by parts, it follows that critical points corresponds to closed geodesics.
\begin{lem}
$\gamma \in \Lambda  \MM$ is a closed geodesic (or a constant map) if
and only if  it
is a critical point of $E$, namely $\displaystyle
dE(\gamma)[\xi]=0  \textrm{ for all  } \xi \in T_\gamma \Lambda  M.$
\end{lem}
\begin{proof}
 For the proof we refer the interested reader to \cite[Theorem 1.3.11]{Kli78}.
\end{proof}
If $\gamma \in \Lambda    M$ is a
non-constant closed geodesic in $( \MM,\mathfrak g)$ then  the
second variation of the geodesic energy functional $E$  at $\gamma$
is the {\em index form\/} $\mathfrak I_\gamma$ given by
\begin{equation}\label{eq:index_form}
\mathfrak I_\gamma[\xi, \eta]= \int_0^T \Big[\mathfrak g\left(\Ddt \xi (t) ,
  \Ddt \eta(t)\right )+  \mathfrak g\big (\mathcal R(\dot \gamma(t) , \xi(t)) \dot
\gamma(t), \eta(t)\big)
\Big]dt.
\end{equation}
It is readily seen that
$\mathfrak I_\gamma$ is a bounded symmetric bilinear form on $T_{\gamma}\Lambda M$
whose associated quadratic form will be denoted by $\mathfrak  Q_\gamma$.
Let $ \mathcal H^\perp_{\gamma}\subset \mathcal H_{\gamma}$ be the
closed (real) codimensional one subspace of all $T$-periodic $W^{1,2}$-vector fields
along $\gamma$  that are everywhere orthogonal to
$\dot
\gamma$. We observe that  $\displaystyle
\dim \ker \mathfrak Q_\gamma = \dim\ker \left(\mathfrak Q_\gamma|_{\mathcal
H^\perp_{\gamma}}\right)+1 $.  We term  {\em nullity of the
geodesic $\gamma$\/} the (non-negative) integer defined by 
$\displaystyle n_0(\mathfrak Q_\gamma)\= \dim \ker \left(\mathfrak Q_\gamma|_{\mathcal H^\perp_{\gamma}}\right).$ Thus the nullity of $\gamma$ is defined as the dimension of the space of
periodic Jacobi fields along $\gamma$ that are pointwise $\mathfrak
g$-orthogonal to $\dot \gamma$.

Since $e_i$ is a parallel vector field  along $\gamma$ and if $\xi$ is a Jacobi field,
then, by direct computation we have
\begin{equation}
\frac{d^{2}}{d^{2}t}\mathfrak{g}(\xi(t),\dot{\gamma}(t))
=\mathfrak{g}(\Ddtt \xi  (t),\dot{\gamma}(t))
=\mathfrak{g}(\mathcal R(\dot \gamma (t) , \xi  (t)) \dot \gamma (t),\dot{\gamma}(t))=0,
\end{equation}
so $\mathfrak{g}(\xi(t),\dot{\gamma}(t))=\alpha t+\beta$ for some  $\alpha,\beta \in \R$.
In particular, if $\xi(0),\xi'(0)\in N_{\gamma(0)}M$, it follows that $\alpha=\beta=0$.  which is equivalent to
$\mathfrak{g}(\xi(t),\dot{\gamma}(t))\equiv0$ for any $t\in[0,T]$. By this fact, we infer that
 $N_{\gamma(0)}M\oplus N_{\gamma(0)}M$ is invariant
under the linearized geodesics flow. (As we already observed, being the parallel transport a $\mathfrak g $-isometry).
In particular the Morse-Sturm system  given in Equation \eqref{eq:Space-Original-intro} decouples into a scalar differential
equation (corresponding to the Jacobi field)
 and a differential system in $\R^n$
(corresponding to the restriction of the Jacobi deviation equation to vector fields $\mathfrak g$-orthogonal to $\dot \gamma$).

\section{Geometric and spectral index  of a closed
geodesic}\label{sec:geometric-variational-index}

This section is devoted to show that the geometrical and the spectral index previously introduced are actually well-defined.

 \begin{lem}
 The geometric index of a closed non-lightlike semi-Riemannian geodesic $\gamma$ is well-defined, i.e.
 it is independent on the trivializing parallel frame along $\gamma$.
 \end{lem}
 \begin{proof}
 We prove the result only in the case of closed spacelike geodesics being the timelike case completely analogous.
 Let  $\{f_1(0),\ldots,f_n(0)\}$ be another
 $\mathfrak g$-orthonormal basis of $N_{\gamma(0)}M$
 such that $f_i(0)=\sum_{j=1}^n c_{ij}e_j(0)$ and we let $C\=[c_{ij}]_{i,j=1}^n$.
 Then, by direct computation,
 the Morse-Sturm system given in Equation \eqref{eq:Space-Reduced-space-intro} fits into the following
  \begin{equation}\label{eq:transformed M-system}
-G\ddot u(t)+\widetilde{R}(t) u(t)=0,\qquad t \in [0,T]
  \end{equation}
where $\widetilde{R}(t)=C\widehat{R}(t)\trasp{C}$ and the boundary condition of $u$ is given
by $u(0)=A\trasp{C}u(T)$. We let
 $C_d\=\begin{bmatrix}\traspm{C}&0\\0&C\end{bmatrix}$ and let $\widetilde \Psi$ be the fundamental solution
 of the corresponding Hamiltonian system given by
 $
\dot z(t)= \widetilde{K}(t)z(t),
$
with $\widetilde{K}(t)=C_dK(t)\trasp{C}_d$. It is easy to check that $\widetilde{\Psi}(t)=\traspm{C}_d\Psi(t)$ and that
\[
\iota_1\big(A_d\trasp{C}_d\widetilde{\Psi}(t); t \in [0,T]\big)=\iota_1(A_d\Psi(t); t \in [0,T]\big).
\]
 This concludes the proof.
 \end{proof}
\begin{lem}
For every $c \in [0,1]$ and $s \in [0,+\infty)$, the operator $\mathcal A^\omega_{c,s}$ is
formally selfadjoint.
\end{lem}
\begin{proof}
The proof of this result readily follows by integrating by parts.
\end{proof}
\begin{lem}\label{thm:lemma1}
For any $c_0>0$, there exists a sufficiently large $s_0$ such that  for any  $|c| \leq c_0$ and $s \geq s_0$,
the operator defined by
\[
\mathcal B_{c,s}\= -G \dfrac{d^2}{dt^2} + c \Id + s G, \qquad s \in [0,+\infty)
\]
is non-degenerate (meaning that has a trivial kernel) on 
\[
E^2_\omega([0,T])=\Set{u\in W^{2,2}([0,T], \C^n)| u(0)=\omega A u(T), \  \dot u(0)=\omega A \dot u(T)} \textrm{ for every } \omega\in\, \U.
\]
\end{lem}
\begin{proof}
We consider the Morse-Sturm system
\[
\begin{cases}
-G \ddot u(t) + \big(c \Id + s G\big) u(t)=0, \quad t \in [0,T]\\
u(0)= \omega A u(T) , \quad  \dot u(0)= \omega A \dot u(T).
\end{cases}
\]
Then the corresponding Hamiltonian system is given by
\begin{equation}\label{eq:new-ham}
\dot z(t) = J B_{c,s} z(t), \qquad t \in [0,T]
\end{equation}
where $B_{c,s}\= \begin{bmatrix} G & 0 \\ 0 & -c\Id -sG \end{bmatrix}$
and the boundary condition is given  by $z(0)= A_d z(T)$. Denoting by $\Phi_{c,s}$ the fundamental
solution of the Hamiltonian system, it follows that the non-degeneracy of $\mathcal B_{c,s}$ is equivalent
to the following condition
\[
\det(\omega A_d \Phi_{c,s}(T)-\Id) \neq 0.
\]
If $s$ is sufficiently large, then $s \pm c$ is positive. We set
\[
\lambda \= \sqrt{s+c} \quad \textrm{ and } \quad \mu= \sqrt{s-c}.
\]
By a direct computation, we get
\[
\Phi_{c,s}(t)=
\begin{bmatrix}
\cosh(\lambda\, t)\Id & 0 & \lambda \sinh (\lambda\, t) \Id& 0\\
0 & \cosh( \mu \, t) \Id& 0 & -\mu \sinh(\mu\, t)\Id\\
\dfrac{\sinh(\lambda \, t)}{\lambda}\Id & 0 & \cosh(\lambda \, t)\Id & 0 \\
0 & -\dfrac{\sinh(\mu t)}{\mu}\Id & 0 & \cosh(\mu \, t)\Id
\end{bmatrix}, \qquad t \in [0,T].
\]
We observe that
\begin{multline}
(\omega A_d)^{-1}= \bar \omega A_d^{-1}= \bar \omega
\begin{bmatrix} \trasp{A} & 0 \\0 & A^{-1}\end{bmatrix}=
\begin{bmatrix} \bar \omega \trasp{A} & 0 \\0 & \bar \omega A^{-1}\end{bmatrix} \textrm{ and } \\
(\bar \omega \trasp{A}) G (\omega A) = G \Rightarrow \bar \omega \trasp{A} = G (\omega A)^{-1} G= G(\bar \omega A^{-1}) G.
\end{multline}
We let 
\[
\bar \omega A^{-1}\= \begin{bmatrix} P & Q \\ R & S \end{bmatrix}, \quad \textrm{ we immediately get that }  \bar \omega \trasp{A}= \begin{bmatrix} P &- Q \\ -R & S \end{bmatrix}.
 \]
In shorthand notation, we let

\begin{multline}
U\= \begin{bmatrix}
\cosh(\lambda T)\Id - P & Q \\
R &\cosh(\mu T)\Id - S
\end{bmatrix}, \quad V\= \begin{bmatrix}
\lambda \sinh(\lambda T)\Id  & 0 \\
0 &  -\mu\sinh(\mu T)\Id
\end{bmatrix}\\
X\= \begin{bmatrix}
\dfrac{\sinh(\lambda T)}{\lambda}\Id & 0 \\
0 & -\dfrac{\sinh(\mu T)}{\mu}\Id
\end{bmatrix},\quad Y\= \begin{bmatrix}
\cosh(\lambda T)\Id -P  & -Q \\
-R & \cosh(\mu T)\Id - S
\end{bmatrix}
 \end{multline}
and we observe that $U,V,X,Y$ are $n \times  n $ matrices. 
 Then $\displaystyle
 \Phi_{c,s}(T) - (\omega A_d)^{-1}= \begin{bmatrix}
 U & V\\ X & Y
 \end{bmatrix}$
  and hence
 \begin{align}
  \det \big(\omega A_d\Phi_{c,s}(T) - \Id \big)&= \det(\omega A_d) \cdot \det\big(\Phi_{c,s}(T)-(\omega A_d)^{-1}\big)\\
 &= \omega^{2n} \cdot \det A \cdot \det \traspm{A} \det\begin{bmatrix}U&V\\X&Y\end{bmatrix}\\
 &=(-1)^{n^2} \omega^{2n}\cdot \det \begin{bmatrix}U&V\\X&Y\end{bmatrix}=
(-1)^n \omega^{2n} \cdot \det\big( V\cdot(X-Y V^{-1} U)\big).
 \end{align}
 By a straightforward calculation, we get
  \begin{multline}
 V\cdot(X-Y V^{-1} U)\\=
 \begin{bmatrix}
-\big(\cosh(\lambda T) -P\big)^2-C_{\lambda,\mu}QR+\sinh^{2}(\lambda T)&
-\big(\cosh(\lambda T)-P\big)Q-C_{\lambda,\mu} Q\big(\cosh(\mu T)-S\big)\\
-C_{\lambda,\mu}^{-1} R \big(\cosh(\lambda T)-P\big)-\big(\cosh(\mu T)-S\big)R&
-\big(\cosh(\mu T)-S)^2-C_{\lambda,\mu}^{-1}RQ+\sinh^2(\mu T)
 \end{bmatrix}
 \end{multline}
 where we set $C_{\lambda,\mu}\=\dfrac{\lambda \sinh(\lambda T)}{\mu\sinh(\mu T)}$. Thus asymptotically, we get 
 following behavior
 \begin{multline}
 -\big(\cosh(\lambda T) -P\big)^2-C_{\lambda,\mu}QR+\sinh^{2}(\lambda T )= -1 + 2 \cosh(\lambda T) P-P^2
  - C_{\lambda,\mu} QR \sim_{+\infty} 2\cosh(\lambda T)P\\
  -\big(\cosh(\lambda T)-P\big)Q-C_{\lambda,\mu} Q\big(\cosh(\mu T)-S\big) \sim_{+\infty}  -2\cosh(\lambda T)Q\\
  -C_{\lambda,\mu}^{-1} R \big(\cosh(\lambda T)-P\big)-\big(\cosh(\mu T)-S\big)R \sim_{+\infty} -2\cosh(\lambda T)R\\
  -\big(\cosh(\mu T)-S)^2-C_{\lambda,\mu}^{-1}RQ+\sinh^2(\mu T)\sim_{+\infty} 2\cosh(\lambda T)S.
   \end{multline}
   Summing up, we have 
   \begin{equation}\label{eq:stima}
    V\cdot(X-Y V^{-1} U)\sim_{+\infty} 2 \cosh(\lambda T) \begin{bmatrix} P & -Q\\ -R & S \end{bmatrix} = 2\cosh(\lambda T) \bar \omega \trasp{A} .
   \end{equation}
   By taking into account Equation \eqref{eq:stima},  it holds that
   \begin{align}\label{eq:stima2}
   \det \big(\omega A_d \Phi_{c,s}(T)-\Id\big)&=(-1)^n \omega^{2n} \det\big(V(X-YV^{-1}U)\big) &\sim_{+\infty} (-1)^n\omega ^{2n}\det\big(2 \cosh(\lambda T)\bar \omega \trasp{A}\big)\\
   &= (-1)^n \omega^{2n} \big(2\cosh(\lambda T)\big)^n\bar\omega^n \det A\\& \neq 0.
   \end{align}
 By Equation \eqref{eq:stima2} the thesis readily follows. This concludes the proof.
 \end{proof}
In order to prove the non-degeneracy of the operator $\mathcal A_{1,s}^\omega$  for $s$ sufficiently large, we need the following stability result proved by Kato.
\begin{lem}\label{thm:lemma22}
Let $T$ be a selfadjoint operator and $A$ be symmetric. Then the operator $S\=T+A$ is
selfadjoint and
\[
\dist\big(\sigma(S), \sigma(T)\big) \leq \norm{A},
\]
where $dist(\cdot,\cdot)$ is the Hausdorff distance.
\end{lem}
\begin{proof}
For the proof, we refer the interested reader, to \cite[pag. 291]{Kat80}.
\end{proof}
\begin{cor}\label{thm:lemma3}
Let $c_0>0$ be such that $\norm{\widehat R}_{\mathscr L(\R^n)} \leq c_0$  and $s_0$ be the number related to $c_0$ by Lemma \ref{thm:lemma1}. Then for any $s \geq s_0$ it holds
\[
\mathcal A_{1,s}^\omega\=-G \dfrac{d^2}{dt^2} +\widehat R + s G
\]
is non-degenerate.
\end{cor}
\begin{proof}
The proof of this result  follows by Lemma \ref{thm:lemma1} and Lemma \ref{thm:lemma22}, just by setting (with a slight abuse of notation)
\[
T\= -G \dfrac{d^2}{dt^2}  + s G \textrm{ and } A\=\widehat R.
\]
This concludes the proof.
\end{proof}
 \begin{rem}
We observe that in the Riemannian case, directly proof of this result can be easily conceived. However in the semi-Riemannian setting the abstract  way (which works in the Riemannian world) for simplifying this proof breaks-down. The obstruction to carry over that proof  is essentially based on the fact that the matrix $A$ coming from the trivialization is a $\mathfrak g$-orthogonal and not just an orthogonal matrix, like in the Riemannain case.
 \end{rem}

\section{A generalized spectral flow formula}\label{sec:index-theorem}

This section is devoted to the relation intertwining  the spectral index
and the geometric index. As a direct consequence we get a spectral flow
formula involving the $\omega$-spectral index and the geometric index.
The following spectral flow formula holds.
\begin{prop}\label{thm:lemma4}
Let $s_0$ be given in Lemma \ref{thm:lemma1}. Then, we have
\[
\spfl(\mathcal A^{\omega}_{1,s}; s \in [0,s_0]) =-\iota_1\big(\omega A_d\Psi_{1,s}(T); s \in [0,s_0]\big).
\]
\end{prop}
\begin{proof}
For a given matrix $M\in L(\C^n)$, we denote its graph by $\Gr(M)\=\{\trasp{(x,\trasp{(Mx)})}\ |\ x\in \C^n\}$.  
By Definition \ref{def:Maslov-index} of  the $\iCLM$-index, we have 
\begin{align}
\iota_1\big(\omega A_d\Psi_{1,s}(T); s \in [0,s_0]\big)&= \iota_{\bar \omega}\big( A_d\Psi_{1,s}(T); s \in [0,s_0]\big)=
 \iota_{\omega}\big( A_d\Psi_{1,s}(T); s \in [0,s_0]\big)\\
 &= \iCLM\big(\Gr(\omega \Id), \Gr\big(A_d\Psi_{1,s}(T) \big); s \in [0,s_0]\big),
\end{align}
where the last equality follows from Equation \eqref{eq:non-citata} and Proposition \ref{prop:L-C}. By invoking \cite[Lemma 2.3 and Lemma 2.4]{HS09}, we have
\begin{equation}\label{eq:caju}
\iCLM\big(\Gr(\omega \Id), \Gr\big(A_d\Psi_{1,s}(T) \big); s \in [0,s_0]\big)= -\spfl(\mathcal D^\omega_{1,s}; s \in [0,s_0])
\end{equation}
where $\mathcal D_{c,s}\=-J \dfrac{d}{dt}- D_{c,s}(t); s \in [0,s_0]$ on the Sobolev space
\[
E^1_\omega([0,T])\=\Set{u\in W^{1,2}([0,T], \C^{2n})| u(0)=\omega A_d u(T)}.
\]
In order to concludes the proof, it is enough to prove that
\[
\spfl(\mathcal A^\omega_{1,s}; s \in [0,s_0])=
\spfl(\mathcal D^\omega_{1,s}; s \in [0,s_0]).
\]
For, we start to assume that both paths are regular in the sense specified in Appendix \ref{subsec:spectral-flow}.
Under this regularity assumption it is enough to show that the local contribution of the
spectral flow of both paths $\mathcal A^\omega_{1,s}$ and $\mathcal D_{1,s}^\omega$ coincide. Let $s_* \in [0,s_0]$  and
we start to observe that
\[
\dim \ker \mathcal A_{1,s_*} \neq \{0\} \iff \dim\ker \mathcal D_{1,s_*} \neq \{0\}.
\]
In particular $s_*$ is a crossing instant for $s\mapsto\mathcal A_{1,s}$ if and only if it is a crossing instant for $s\mapsto\mathcal D_{1,s}$.
Let $s_*$ be a crossing instant. By a direct computation, it follows that  the crossing form of $\mathcal A$ at $s_*$ is the quadratic form
\begin{equation}\label{eq:crossing1}
\Gamma(\mathcal A_{1,s}, s_*): \ker \mathcal A_{1,s} \to \R \textrm{ defined by }
\Gamma(\mathcal A_{1,s}, s_*)[u]= \int_0^T \langle Gu, u\rangle\, dt, \quad u \in \ker \mathcal A_{1,s_*}.
\end{equation}
The crossing form of $\mathcal D_{1,s}$ at $s_*$ is given by
\begin{multline}\label{eq:crossing2}
\Gamma(\mathcal D_{1,s}, s_*): \ker \mathcal D_{1,s} \to \R \textrm{ defined by }
\Gamma(\mathcal D_{1,s}, s_*)[w]= \int_0^T \left\langle \begin{bmatrix}0 & 0 \\ 0 & G\end{bmatrix}\begin{bmatrix} G \dot w\\ w\end{bmatrix},
\begin{bmatrix} G \dot w\\ w\end{bmatrix}\right \rangle\, dt \\
=\int_0^T \langle Gw, w \rangle\, dt \quad w \in \ker \mathcal D_{1,s_*}.
\end{multline}
By Equations \eqref{eq:crossing1}-\eqref{eq:crossing2}, we immediately conclude that the crossing forms
coincide and hence also their signatures. In particular the local contribution at $s_*$ of both paths to the spectral flow coincide and
this concludes the proof, under the assumption that $s\mapsto\mathcal A_{1,s}$ and $s\mapsto\mathcal D_{1,s}$ are regular.

In order to concludes the proof in the general case (i.e. for non-regular paths), we observe that by standard perturbation results there exists $\varepsilon >0$ sufficiently small
(cf. Appendix \ref{subsec:spectral-flow} and references therein) such that the perturbed path
$s \mapsto\mathcal A_{1,s}^\varepsilon\= \mathcal A_{1,s}+ \varepsilon \Id$ is regular and, by the homotopy invariance property with fixed ends,
it  has the same spectral flow of $s\mapsto \mathcal A_{1,s}$. Through the perturbed path $\mathcal A_{1,s}^\varepsilon$, we can define the perturbed
path of first order operators given by
\[
\mathcal D_{1,s}^\varepsilon \= -J \dfrac{d}{dt} - D_{1,s}^\varepsilon(t)
\]
where $D_{1,s}^\varepsilon(t)\= \begin{bmatrix}
G & 0 \\ 0 & -\widehat R(t)-sG +\varepsilon \Id
\end{bmatrix}$. Repeating ad verbatim the same arguments for  the perturbed paths, we get the thesis. This concludes the proof.
\end{proof}
Our next step is to compute the integer $\iota_1\big(\omega A_d \Psi_{0,s_0}(t); t \in [0,T]\big)$. Before introducing a technical
result needed for this computation, we observe that
\begin{align}\label{eq:suigrafici}
\iota_1\big(\omega A_d \Psi_{0,s_0}(t); t \in [0,T]\big)&= \iota_\omega\big( A_d \Psi_{0,s_0}(t); t \in [0,T]\big)\\
&=\iCLM\Big(\Gr\big(\omega A_d^{-1}\big), \Gr\big(\Psi_{0,s_0}(t)\big); t \in [0,T]\Big).
\end{align}
The  idea for computing the (RHS) of Equation \eqref{eq:suigrafici} is to transform the path of graphs of a symplectic matrix into a path of graphs of
a symmetric matrices. In this  way the computation of the $\iCLM$ index can be performed through the spectral flow.

\begin{lem}\label{thm:lemma0-2}
Under the previous notation, we have
\[
\iCLM\Big(\Gr\big(\omega A_d^{-1}\big), \Gr\big(\Psi_{0,s_0}(t)\big); t \in [0,T]\Big)= \iCLM\Big(L_D, \Gr\big(M_{s_0}(t)\big); t \in [0,T]\Big)
\]
where $L_D= \Set{\begin{bmatrix}x\\ 0\end{bmatrix}| x \in \C^{2n}}$ is the (horizontal) Dirichlet Lagrangian subspace
 and $t \mapsto M_{s_0}(t)$ is the path pointwise defined by
\[
M_{s_0}(t)\=\begin{bmatrix}
\dfrac{\sinh(\sqrt{s_0}t)}{\sqrt{s_0}\cosh(\sqrt{s_0}t)} G & \dfrac{I}{\cosh(\sqrt{s_0}t)}-\omega A^{-1}\\
\dfrac{I}{\cosh(\sqrt{s_0}t)}-\bar \omega \traspm{A} & -\sqrt{s_0}\dfrac{\sinh(\sqrt{s_0}t)}{\cosh(\sqrt{s_0}t)} G
\end{bmatrix}.
\]
\end{lem}
\begin{rem}\label{rem:idimu}
We observe that  if $P\in \Sp(2n)$, then
\[
\Gr(P)=\Set{\begin{bmatrix}x \\ Px\end{bmatrix}|x \in \C^{2n}}
\]
is a Lagrangian subspace of $\C^{2n} \oplus \C^{2n}$ under the symplectic form $-J \oplus J$, where $J$ is the standard symplectic matrix of $\C^{2n}$. We recall
that a {\em Lagrangian  frame}  for a Lagrangian subspace $L$ is an injective linear map $Z: \C^n \to \C^{2n}$  whose image is $L$. Such a frame has the form $\displaystyle Z= \trasp{(X,Y)}$
where $X,Y$ are $n \times n$-matrices and $Y^*X=X^*Y$,
where $*$ denotes the conjugate transpose.
In the special case in which $L$ is a graph of a symmetric matrix, then $X=\Id$ and the  relation $Y^*X=X^*Y$ trivially holds. We observe also that if $X$ is
invertible, then another Lagrangian frame for $L$ with respect to which it is a graph, is given by $\displaystyle
W= \trasp{ (\Id, YX^{-1})}$.  In fact, being $L = \im(Z)=\Set{\begin{bmatrix}Xu\\ Yu\end{bmatrix}|u \in \C^n}$, it follows that by changing coordinates by
setting $u=X^{-1}w$, we get $L = \im(Z)=\Set{\begin{bmatrix}w\\ YX^{-1}w\end{bmatrix}|u \in \C^n}$. We start to observe that
if $L$ is a Lagrangian subspace  with respect to a  symplectic form $ \widehat \omega (\cdot, \cdot)\=\langle \widehat J \cdot, \cdot \rangle$,
then  for any orthogonal matrix $S$, the subspace $S^{-1}L$ is a Lagrangian subspace with respect to the symplectic
form  $\widetilde \omega (\cdot, \cdot)\=\langle \widetilde J \cdot, \cdot \rangle$
represented  by $\widetilde J\= \trasp{S}\widehat J S$. Moreover, if  $\begin{bmatrix}Y \\X\end{bmatrix}$ is a Lagrangian frame for $L$, then
$S^{-1}\begin{bmatrix}Y \\X\end{bmatrix}$ is a Lagrangian frame for $S^{-1}L$.
\end{rem}
\begin{proof}
We start to define the orthogonal matrix $S$ given by
$\displaystyle
S\=
\begin{bmatrix}
0 & 0 & \Id & 0 \\
0 & \Id  & 0 & 0 \\
0 & 0 & 0 & \Id\\
\Id  & 0 & 0 & 0
\end{bmatrix}.
$
Identifying  $\Gr\big(\Psi_{0,s_0}(t)\big)$ with its Lagrangian frame $\begin{bmatrix}
\Id & 0 \\
0 & \Id\\
\cosh(\sqrt{s_0}t) \Id& \sqrt{s_0} \sinh(\sqrt{s_0} t) G\\
\dfrac{1}{\sqrt{s_0}} \sinh(\sqrt{s_0}t) G & \cosh(\sqrt{s_0} t)\Id
\end{bmatrix}$, we get
\begin{equation}
S\Gr\big(\Psi_{0,s_0}(t)\big)=
\begin{bmatrix}
\cosh(\sqrt{s_0}t) \Id& \sqrt{s_0} \sinh(\sqrt{s_0} t) G\\
0 & \Id\\
\dfrac{1}{\sqrt{s_0}} \sinh(\sqrt{s_0}t) G & \cosh(\sqrt{s_0} t)\Id\\
\Id & 0 \\
\end{bmatrix}
\end{equation}
is a Lagrangian subspace with respect to the standard symplectic form $\widetilde J$ of $\C^{2n}\oplus \C^{2n}$. To do so,
it is enough to observe that by a direct computation, we have
\begin{equation}
\widetilde J = S \big(-J \oplus J\big)\trasp{S}=
\begin{bmatrix}
0 & 0 &-\Id&0 \\
0 & 0& 0 & -\Id \\
\Id & 0 & 0 & 0 \\
0 & \Id & 0 & 0\\
\end{bmatrix}.
\end{equation}
By taking into account Remark \ref{rem:idimu} and by observing that  the matrix
\[
\begin{bmatrix}
\cosh(\sqrt{s_0}t) \Id& \sqrt{s_0} \sinh(\sqrt{s_0} t) G\\
0 & \Id
\end{bmatrix}
\]
is invertible, we can re-write the Lagrangian subspace $S\Gr\big(\Psi_{0,s_0}(t)\big)$ as the graph of a symmetric matrix, simply by
change the Lagrangian  frame. To do so, we observe that
\begin{multline}
\begin{bmatrix}
\dfrac{1}{\sqrt{s_0}} \sinh(\sqrt{s_0}t) G & \cosh(\sqrt{s_0} t)\Id\\
\Id & 0 \\
\end{bmatrix}\cdot
\begin{bmatrix}
\cosh(\sqrt{s_0}t) \Id& \sqrt{s_0} \sinh(\sqrt{s_0} t) G\\
0 & \Id\\
\end{bmatrix}^{-1}\\
=\begin{bmatrix}
\dfrac{\sinh(\sqrt{s_0}t)}{\sqrt{s_0}\cosh(\sqrt{s_0}t)} G & \dfrac{1}{\cosh(\sqrt{s_0}t)}\Id\\
\dfrac{1}{\cosh(\sqrt{s_0}t)} \Id& -\sqrt{s_0}\dfrac{\sinh(\sqrt{s_0}t)}{\cosh(\sqrt{s_0}t)} G
\end{bmatrix}
\end{multline}
Thus the Lagrangian frame for $S\Gr\big(\omega \Psi_{0,s_0}(t)\big)$ fits into the following
\[
\begin{bmatrix}
\Id&0\\0&\Id\\
\dfrac{\sinh(\sqrt{s_0}t)}{\sqrt{s_0}\cosh(\sqrt{s_0}t)} G & \dfrac{1}{\cosh(\sqrt{s_0}t)}\Id\\
\dfrac{1}{\cosh(\sqrt{s_0}t)} \Id& -\sqrt{s_0}\dfrac{\sinh(\sqrt{s_0}t)}{\cosh(\sqrt{s_0}t)} G
\end{bmatrix}
\]
By the very same computations for the Lagrangian subspace $\Gr\big(\omega A_d^{-1}\big)$, we get
\[
S\Gr\big(\omega A_d^{-1}\big)= \begin{bmatrix}
0 & 0 & \Id & 0 \\
0 & \Id& 0 & 0 \\
0 & 0 & 0 & \Id\\
\Id & 0 & 0 & 0
\end{bmatrix}\cdot
\begin{bmatrix}
\Id & 0 \\
0 & \Id\\
\omega \trasp{A}&0\\
0 &\omega A^{-1}
\end{bmatrix} =
\begin{bmatrix}
\omega \trasp{A} &0\\
0 & \Id\\
0 & \omega A^{-1}\\
\Id & 0 \\
\end{bmatrix}.
\]
In this way the Lagrangian subspace $\Gr\big(\omega A_d^{-1}\big)$ can be re-written as follows
\[
\begin{bmatrix}
\Id&0\\
0&\Id\\
0&\omega A^{-1}\\
\bar \omega \traspm{A} & 0
\end{bmatrix}
\]
Summing up, we proved that
\[
\iCLM\Big(\Gr\big(\omega A_d^{-1}\big), \Gr\big(\Psi_{0,s_0}(t)\big); t \in [0,T]\Big)= \iCLM\Big(\Gr\big(Z_\omega\big), \Gr\big(Z_{s_0}(t)\big); t \in [0,T]\Big)
\]
where
\[
Z_\omega \=\begin{bmatrix}
0&\omega A^{-1}\\
\bar \omega \traspm{A} & 0
\end{bmatrix}\quad \textrm{ and } \quad Z_{s_0}(t)\=
\begin{bmatrix}
\dfrac{\sinh(\sqrt{s_0}t)}{\sqrt{s_0}\cosh(\sqrt{s_0}t)} G & \dfrac{1}{\cosh(\sqrt{s_0}t)}\Id\\
\dfrac{1}{\cosh(\sqrt{s_0}t)} \Id& -\sqrt{s_0}\dfrac{\sinh(\sqrt{s_0}t)}{\cosh(\sqrt{s_0}t)} G
\end{bmatrix}.
\]
Now the conclusion follows by
\[
\iCLM\Big(\Gr\big(Z_\omega\big), \Gr\big(Z_{s_0}(t)\big); t \in [0,T]\Big)= \iCLM\Big(L_D, \Gr\big(Z_{s_0}(t)- Z_\omega\big); t \in [0,T]\Big)
\]
once observed that $Z_{s_0}(t)- Z_\omega=:M_{s_0}(t)$. This concludes the proof.
\end{proof}
\begin{lem}\label{thm:lemma2-2}
For every $t_0 \in (0,T]$, there exists  $s_0>0$ sufficiently large such that $M_{s_0}(t)$ is non-degenerate and
\[
 \ncind{M_{s_0}(t)}= \nind{M_{s_0}(t)} \textrm{ for any } t \in [t_0, T]
\]
 where $\ncind{*}$ and  $\nind{*}$ denotes respectively  the coindex and the index of $*$.
\end{lem}
\begin{proof}
For every $t \neq 0$ the matrix $\dfrac{\sinh(\sqrt{s_0}t)}{\sqrt{s_0}\cosh(\sqrt{s_0}t)} G$ is invertible. We let
\[
N(t)\=\begin{bmatrix}
\Id & -\dfrac{\sqrt{s_0}\cosh(\sqrt{s_0}t)}{\sinh(\sqrt{s_0}t)} G \left(\dfrac{I}{\cosh(\sqrt{s_0}t)}-\omega A^{-1}\right)\\
0 & \Id
\end{bmatrix}
\]
and we observe that
\begin{multline}
N^*(t)M_{s_0}(t)N(t)\\=
\begin{bmatrix}
\dfrac{1}{\sqrt{s_0}}\tanh(\sqrt{s_0t}) G & 0 \\
0 &- \sqrt{s_0}\dfrac{1}{\sinh(\sqrt{s_0}t) }\Big(2 \cosh(\sqrt{s_0} t)G-\bar \omega \traspm{A} G -\omega G A^{-1}\Big)
\end{bmatrix}.
\end{multline}
Thus
\begin{multline}
\det M_{s_0}(t)=\det\big(N^*(t)M_{s_0}(t)N(t)\big)=
\det\left(-2G+ \dfrac{1}{\cosh(\sqrt{s_0}t) }\big(\bar \omega \traspm{A} G + \omega G A^{-1}\big)\right).
\end{multline}
Since for $t \in [t_0, T]$ it holds that
\[
\det M_{s_0}(t) \sim_{+\infty} \det(-2 G )\neq 0,
\]
then $M_{s_0}(t)$ is non-degenerate for $t \in [t_0,T]$. We observe also that
\[
N^*(t)M_{s_0}(t)N(t)\sim_{+\infty}
\begin{bmatrix}
\dfrac{1}{\sqrt{s_0}}\tanh(\sqrt{s_0t}) G & 0 \\
0 & - 2\sqrt{s_0}\coth(\sqrt{s_0} t) G
\end{bmatrix}.
\]
Since the index and the coindex of
the two matrices
$M_{s_0}(t)$ and $N^*(t)M_{s_0}(t)N(t)$ agree. This concludes the proof.
\end{proof}

\begin{lem}\label{thm:lemma1-2}
For any complex matrix $B$, we let
\[
M\= \begin{bmatrix}
0 & B\\ B^* & 0
\end{bmatrix}
\]
where we denoted by $B^*$ the conjugate transpose of $B$. Then we have
\[
\nnull{M}= 2 \dim \ker B\quad \textrm{ and }\quad \ncind{M}= \nind{M}
\]
and where $\nnull{*}$ denotes the nullity of $*$.
\end{lem}
\begin{proof}
The proof of this result, readily follows by a direct computation.
\end{proof}

\begin{prop}\label{thm:lemma5}
Under the previous notation, for every $\omega \in \U$, we have
\[
\iota_1(\omega A_d \Psi_{0, s_0}(t); t \in [0,T])= \dim \ker (A-\omega \Id).
\]
\end{prop}
\begin{proof}
We start to observe that
\[\iota_1(\omega A_d \Psi_{0, s_0}(t); t \in [0,T])= \iCLM\Big(\Gr\big(\omega A_d^{-1}\big), \Gr\big(\Psi_{0,s_0}(t)\big); t \in [0,T]\Big).
\]
Then in order to conclude, it
is enough to compute $ \iCLM\Big(\Gr\big(\omega A_d^{-1}\big), \Gr\big(\Psi_{0,s_0}(t)\big); t \in [0,T]\Big)$. By Lemma \ref{thm:lemma0-2}, we have
\[
 \iCLM\Big(\Gr\big(\omega A_d^{-1}\big), \Gr\big(\Psi_{0,s_0}(t)\big); t \in [0,T]\Big)= \iCLM\Big(L_D, \Gr\big(M_{s_0}(t)\big); t \in [0,T]\Big)
\]
 $M_{s_0}(t)$ has been defined in Lemma \ref{thm:lemma1-2}. Moreover
for $t=0$, the matrix $M_{s_0}(t)$ reduces to
\[
M_{s_0}(0)= \begin{bmatrix}
0 & \Id -\omega A^{-1}\\
\Id-\bar\omega \traspm{A} & 0
\end{bmatrix}.
\]
By Lemma \ref{thm:lemma1-2}, we have
\[
\nnull{M_{s_0}(0)}= 2 \dim \ker(A-\omega \Id), \qquad \ncind{M_{s_0}(0)}= \nind{M_{s_0}(0)}.
\]
By Lemma \ref{thm:lemma2-2}, there exists $t_0$ and $s_0$ such that $M_{s_0}(t)$ is non-degenerate for $t \in [t_0,T]$ and
$ \ncind{M_{s_0}(t)}= \nind{M_{s_0}(t)}$ for every $t \in [t_0,T]$. By this, we infer that
\[
\spfl(M_{s_0}(t); t \in [0,T]) = \dim \ker (A-\omega \Id).
\]
In order to conclude, it is enough  to observe that
\[
\iCLM\Big(L_D, \Gr\big(M_{s_0}(t)\big); t \in [0,T]\Big)= \spfl(M_{s_0}(t); t \in [0,T]).
\]
This concludes the proof.
\end{proof}
By using Proposition \ref{thm:lemma5} we are now in position to prove Theorem \ref{thm:Lemma6}.
\paragraph{ Proof of Theorem \ref{thm:Lemma6}.}For $c \in [0,1]$ and $s \in [0,s_0]$
we start to consider the path $s \mapsto  \mathcal A^\omega_{c,s}$ given in Equation \eqref{eq:Acs}, namely
\begin{equation}
\mathcal A^\omega_{c,s}=-G \dfrac{d^2}{dt^2} + c \widehat R(t)+ sG, \qquad t \in [0,T]
\end{equation}
on  the Hilbert  space $E^2_\omega([0,T])$ and defined in Equation \eqref{spazio-Acs}. The corresponding Morse-Sturm
system is given by
\begin{equation}
-G \ddot u(t) + \big(c \widehat R(t) + sG\big) u(t)=0, \qquad t \in [0,T]
\end{equation}
and the associated Hamiltonian system has been  given in Equation  \eqref{eq:Dcs}; i.e.
\begin{equation}
\dot z(t) = J D_{c,s}(t)z(t), \qquad t \in [0,T]
\end{equation}
whose fundamental solution has been  denoted by $s \mapsto \Psi_{c,s}(t)$.
\begin{figure}[ht]
 \centering
 \includegraphics[scale=0.2]{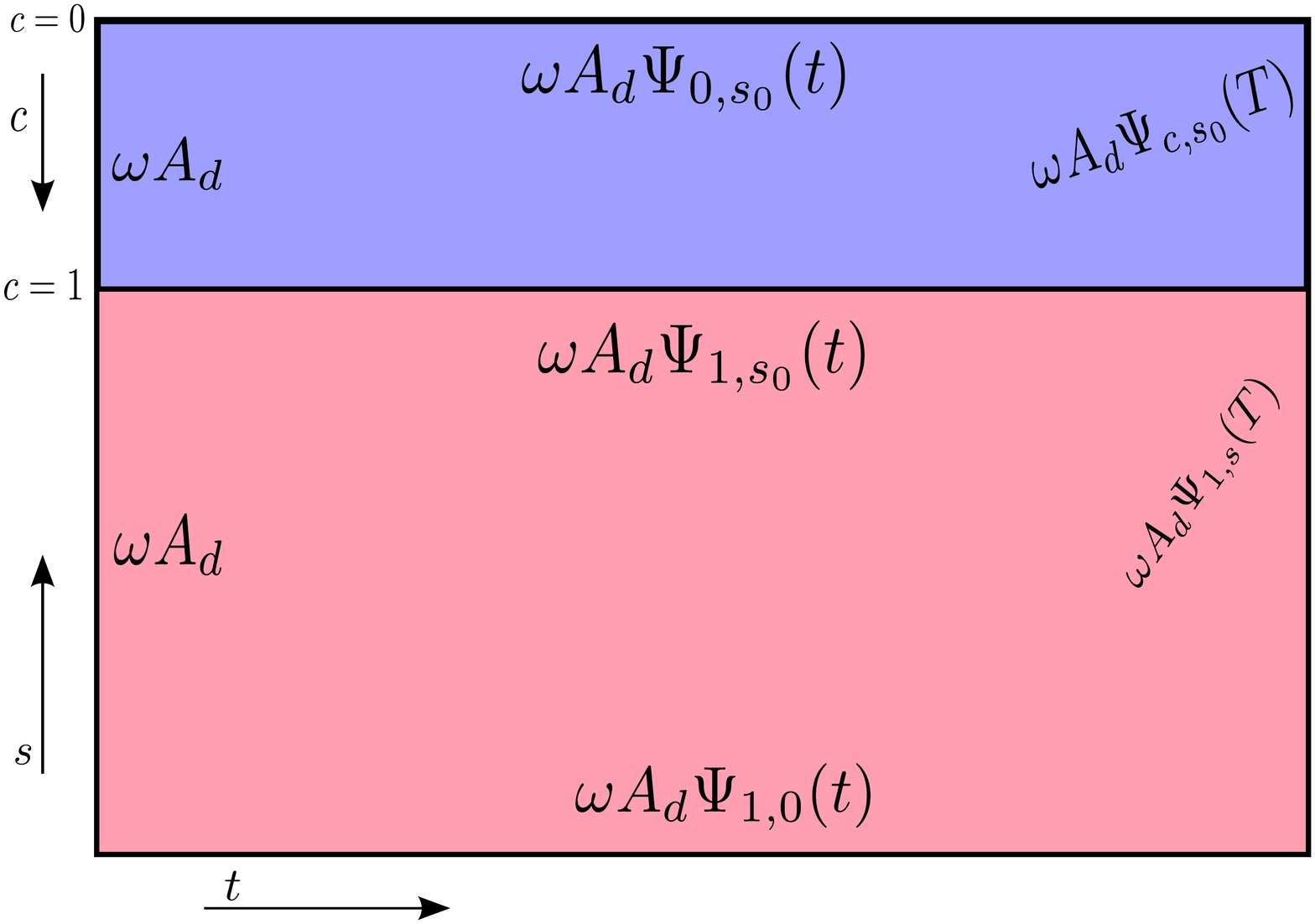}
 \caption{In the blue rectangle (above) is depicted the homotopy with respect to the parameter $c$ whilst in the pink rectangle
 (below) the homotopy with respect to the $s$ parameter.}\label{fig:omotopia}
\end{figure}
In order to prove the result we observe that
\begin{enumerate}
\item $\omega A_d \Psi_{c,s}(0)=\omega A_d$ being $\Psi_{c,s}(0)=\Id$.  This in particular implies that
\[
\iota_1\big(\omega A_d \Psi_{1,s}(0);s \in [0,s_0]\big)=0 \quad \textrm{ and } \quad \iota_1\big(\omega A_d \Psi_{c,s}(0);c \in [0,1]\big)=0.
\]
\item By Corollary \ref{thm:lemma3}, if $s_0$  is sufficiently large, then $\mathcal A_{c,s_0}$ is non-degenerate for every $c \in [0,1]$.
This in particular implies that
\[
\iota_1\big(\omega A_d \Psi_{c, s_0}(T);c \in [0,1]\big) =0.
\]
\item For $c=0$, by Proposition \ref{thm:lemma5}, we infer that
\[
\iota_1(\omega A_d \Psi_{0, s_0}(t); t \in [0,T])= \dim \ker (A-\omega \Id).
\]
\item By the homotopy invariance property of the Maslov-type index, we have
\begin{multline}\label{eq:dainvocare}
\iota_1\big(\omega A_d\Psi_{1,0}(t); t \in [0,T]\big) +\iota_1\big(\omega A_d\Psi_{1,s}(T); s \in [0,s_0]\big)=
\iota_1\big(\omega A_d\Psi_{1,s_0}(t); t \in [0,T]\big) \\
\iota_1\big(\omega A_d\Psi_{0,s_0}(t); t \in [0,T]\big) = \iota_1\big(\omega A_d\Psi_{1,s_0}(t); t \in [0,T]\big) .
\end{multline}
By Proposition \ref{thm:lemma4}, it holds that
\[
\spfl(\mathcal A^{\omega}_{1,s}; s \in [0,s_0]) =-\iota_1\big(\omega A_d\Psi_{1,s}(T); s \in [0,s_0]\big).
\]
\end{enumerate}
By taking into account the relations given in Equation \eqref{eq:dainvocare}, we  get that
\[
\iota_1\big(\omega A_d\Psi_{1,0}(t); t \in [0,T]\big) = \spfl(\mathcal A^{\omega}_{1,s}; s \in [0,s_0])
+ \dim \ker \big( A-\omega \Id\big).
\]
This concludes the proof. \qed

\section{Semi-Riemannian Bott-type iteration formula}\label{subsec:Bott-SR}

The goal of this section is to prove the Bott-type iteration formula for semi-Riemannian geodesics.

We start to observe that, since  $\widehat R$ satisfies the condition $\widehat R(T)= \trasp{A} \widehat R(0) A$, we can  extend it on the
interval $[0,mT]$. More precisely, for every $k=1, \dots, m$, we  define  the  associated Morse-Sturm system  as
\begin{equation}\label{eq:m-iteration-MS-system}
\begin{cases}
-G \ddot u+ \widehat R(t) u=0, \qquad t \in [0,mT]\\
u(0)=A^k u(kT) \textrm{ and } \dot u(0)= A^k \dot u(kT).
\end{cases}
\end{equation}
We now consider the Hamiltonian system $
\dot z(t)= JD_{c,s}(t) z(t), \qquad t \in [0,mT]
$
and we define the operator
$
\mathcal A_{c,s}^{(m)}\= -G \dfrac{d^2}{dt^2} + c \widehat R(t) + s G, \qquad t \in [0, mT]
$
on the Sobolev space
\begin{equation}
E_m^k\= \Set{u \in W^{2,2}([0,mT], \R^n)| u(0)= A^ku(k\,T), \dot u(0)= A^k \dot u(k\,T)},\quad \textrm{ for }  k=1, \dots, m.
\end{equation}
Then, by Theorem \ref{thm:Lemma6}, we have
\begin{equation}
\spfl(\mathcal A_{1,s}^{(m)}; s \in[0,s_0]) +\dim \ker (A^m-\Id)= \iota_1(A_d^m \Psi_{1,0}(t); t \in[0,m\,T]).
\end{equation}
By \cite[Theorem 1.1]{LT15}, we infer
\begin{equation}\label{eq:facili}
\iota_1(A_d^m \Psi_{1,0}(t); t \in [0,m\,T]) =\sum_{\omega^m=1} \iota_\omega(A_d \Psi_{1,0}(t); t \in [0,T])
\end{equation}
and we observe that the following equalities hold
\begin{multline}\label{eq:nullity}
\dim\ker (A^m-\Id)=\sum_{\omega^m=1} \dim\ker (A-\omega \Id),\\
\iota_\omega(A_d \Psi_{1,0}(t); t \in [0,T]) = \iota_1(\omega A_d \Psi_{1,0}(t); t \in [0,T]).
\end{multline}
By Proposition \ref{thm:Lemma6},  Equation \eqref{eq:facili} and Equation \eqref{eq:nullity}, we have
\begin{multline}
\spfl(\mathcal A_{1,s}^{(m)}; s \in [0,s_0]) + \sum_{\omega^m=1}\dim \ker (A-\omega \Id)
= \sum _{\omega^m=1} \iota_1(\omega A_d\Psi_{1,0}(t); t \in [0,T])\quad \Rightarrow \\
\spfl(\mathcal A_{1,s}^{(m)}; s \in [0,s_0])= \sum _{\omega^m=1}\Big( \iota_1(\omega A_d\Psi_{1,0}(t); t \in [0,T])-
\dim \ker (A-\omega \Id)\Big)\\= \sum _{\omega^m=1}\spfl(\mathcal A_{1,s}^{(\omega)}; s \in [0,s_0]),
\end{multline}
where the last equality follows by invoking Theorem \ref{thm:Lemma6}.

\section{A linear  instability criterion}\label{thm:maintheorem}

This Section  is devoted to the proof of the
{\em instability criteria\/} for closed non-lightlike semi-Riemannian geodesics.
\begin{note} \label{not:sp-ours}
We set
\begin{multline}
\Sp(2n,\R)^{+}\=\{M\in\Sp(2n,\R)\mid \det(M-I_{2n})>0\} \textrm{  and } \\
\Sp(2n,\R)^{-}\=\{M\in\Sp(2n,\R)\mid \det(M-I_{2n})<0\}.
\end{multline}
\end{note}
	\begin{rem}
We observe that Notation \ref{not:sp-ours} are non consistent with the standard notation used in literature. In fact, as the reader can easily realize by looking at  Subsection \ref{subsec:Maslov}, there is the missing factor  $(-1)^{n-1}$ in the definition of $\Sp(2n, \R)^+$ and of $\Sp(2n, \R)^-$. However the advantage of defining $\Sp(2n, \R)^\pm$ likes in Notation \ref{not:sp-ours} is for simplifying the discussion about the  linear stability.
	\end{rem}
\begin{lem}\label{thm:lemma2}
 Let $T\in \Sp(2n, \R) $ be a linearly stable symplectic matrix. Then, there
exists $\delta >0$ sufficiently small such that $e^{\pm \delta J}T \in \Sp(2n,\R)^+$.
\end{lem}
\begin{proof}
Let  us consider the (smooth) symplectic path pointwise defined by
 $T(\theta)\= e^{-\theta J}T$. By a direct computation we get that
\[
 \trasp{T(\theta)} J \dfrac{d}{d\theta} T(\theta)\Big\vert_{\theta=0} =
\trasp{T}T.
\]
We observe that $\trasp{T}T$ is symmetric and positive semi-definite; moreover  being $T
$ invertible it follows that
$\trasp{T}T$  is actually positive definite. Thus, in particular,
$\iiindex(\trasp{T}T)=0$. By invoking Proposition \ref{thm:lemma3.2} it follows
that there exists $\delta >0$ such that
$T(\pm \delta) \in \Sp(2n, \R)^+ $. This concludes the proof.
\end{proof}
\begin{rem}
We observe that dropping the linear stability assumption on $T$  in Lemma \ref{thm:lemma2},  the perturbed matrix $e^{\pm \delta J}T$ belongs to $\Sp(2n,\R)^*$; however  we  can't a control in which path-connected component it will lie. (We refer the interested reader to  \cite[Equation (6) and (7), pag. 124]{Lon02} for more details).
\end{rem}

\paragraph{ Proof of Theorem \ref{thm:main-deg}.}
Here it is enough to  prove the
contrapositive, namely
\begin{itemize}
 \item if $\gamma$ is linearly stable and oriented then   $\ispec(\gamma) +n$
  is even;
 \item if $\gamma$ is linearly stable and nonoriented then   $\ispec(\gamma)
+n$   is odd.
\end{itemize}
As direct consequence of Proposition \ref{thm:lemma4}, we know that
 $\ispec(\gamma)=-\igeo(A_d\Psi_{1,s}(T); s \in [0,s_0])$.
 In order to concludes the proof, it is enough to consider the parity of $\igeo(A_d\Psi_{1,s}(T); s \in [0,s_0])+n$.
{\bf Case (OR)} If $\gamma$ is oriented, then $\det(A)=1$. Thus, if
$A_d\Psi_{1,0}(T)$ is linear stable, then by Lemma \ref{thm:lemma2},  it follows that
$e^{-\theta J}A_d\Psi_{1,0}(T)\in \Sp(2n,\R)^{+}$. By a direct computation , we infer that  $\det(A_d\Psi_{0,s_0}-I_{2n})=(-1)^n\big(2\cosh(\sqrt{s_0}T)\big)^n\det A$; thus we get 
$\det(A_d\Psi_{0,s_0}-I_{2n})>0$ (resp. $\det(A_d\Psi_{0,s_0}-I_{2n})<0$) iff $n$ is even (resp. odd). Hence, $A_d\Psi_{0,s_0}\in\Sp(2n,\R)^{+}$(resp. $A_d\Psi_{0,s_0}\in\Sp(2n,\R)^{-}$) iff $n$ is even (resp. odd).
By invoking  Corollary \ref{thm:lemma3}, if $s_0$  is sufficiently large, then $\mathcal A_{c,s_0}$ is non-degenerate for every $c \in [0,1]$. Consequently, $\det(A_d\Psi_{c,s_0}-I_{2n})\neq0$ for every $c\in[0,1]$ and in particular if $A_d\Psi_{0,s_0}\in\Sp(2n,\R)^{+}$ then $A_d\Psi_{1,s_0}\in\Sp(2n,\R)^{+}$ too.  By these arguments, we get that $A_d\Psi_{1,s_0}(T)\in
\Sp(2n,\R)^{+}$ (resp. $A_d\Psi_{1,s_0}(T)\in \Sp(2n,\R)^{-}$)  iff $n$ is even (resp. odd). Now, the conclusion follows by taking into account 
Lemma \ref{thm:parity} since   $\igeo(A_d\Psi_{1,s}(T), s \in [0,s_0])$  is even
(resp. odd) or which is
equivalent that $\ispec(\gamma)$ is even (resp. odd) iff $n$ is even (resp. odd). Therefore, we can conclude that
 $n+\ispec(\gamma)$ is always even.

{\bf Case (NOR)} If $\gamma$ is nonoriented then $\det(A)=-1$. Arguing as above, we
get that
 $n+\ispec(\gamma)$ is  odd.\qed

From now on, if not differently stated, the pair $(M,\mathfrak g)$
will denote a $(n+1)$-dimensional Riemannian  (resp. Lorentzian) manifold and $\gamma:[0,T]\rightarrow M$ a closed (resp. timelike closed) geodesic.
Once trivialized the pull-back of the tangent bundle along the closed Riemannian. (resp. timelike Lorentzian) geodesic $\gamma$, the index form
reduces to the symmetric bilinear form
$I:E\times E\rightarrow \R$ given by
\begin{equation}
I(u,v)=\int^T_0\langle\dot{u}(t),\dot{v}(t)\rangle+\langle\widehat{R}(t)u(t),v(t)\rangle dt,
\end{equation}
where $E\=\{u\in W^{1,2}([0,T],\mathbf{R}^n)\mid u(0)=Au(T)\}$, $A\in \OO(n)$ is an orthogonal matrix. We observe that, in this case, as observed in Section  \ref{subsec:Index-Theorem}, the Morse index $\iMor(\gamma)$ of $\gamma$ (meaning the
dimension of the maximal  subspace such that  $I$ is negative definite),  is well-defined.
Given  $\omega\in \U$, we let $E_{\omega}=\Set{u\in W^{1,2}([0,T],\mathbf{C}^n)|u(0)=\omega Au(T)}$ we define the {\em $\omega$-index form\/}  on $E_\omega$
as follows
\begin{equation}\label{eq:omegaindexform}
\iota_\omega(u,v)=\int^T_0\langle\dot{u}(t),\dot{v}(t)\rangle+\langle\widehat{R}(t)u(t),v(t)\rangle dt,
\end{equation}
where $\langle\cdot,\cdot\rangle$ in Equation \eqref{eq:omegaindexform} denotes the standard Hermitian product.
Following \cite{BTZ82} and denoting by
$\iMor(\omega,\gamma)$ the Morse index of $\iota_\omega$ on $E_{\omega}$, we are in position to give the following definition.
\begin{defn}\label{def:splitting-number-geo} ({\bf \cite[pag. 244]{Lon02}\/} )
The splitting numbers of the closed geodesic $\gamma$ at $\omega\in \U$ are defined by
\begin{equation}\label{eq:S-N}
\mathcal{S}^{\pm}(\omega,\gamma)=\lim_{\theta\to 0^\pm}\iMor(\omega e^{\sqrt{-1}\theta},\gamma)-\iMor(\omega,\gamma).
\end{equation}
\end{defn}
Following the discussion given above, to  the geodesic $\gamma$ we associate the
Morse-Sturm system given by
 \begin{equation}\label{eq:R-M-S}
\begin{cases}
-\ddot{u}(t)+\widehat{R}(t)u(t)=0, \qquad t \in [0,T]\\
u(0)=Au(T).
\end{cases}
\end{equation}
By using the Legendre transformation, the second order system given in Equation \eqref{eq:R-M-S} corresponds to the linear
Hamiltonian system
 \begin{equation}\label{eq:R-H-S}
\dot{z}(t)=JB(t)z(t),  \qquad t \in [0,T]\\
\end{equation}
where $B(t)\=\begin{bmatrix}I&0\\0&-\widehat{R}(t)\end{bmatrix}$.
The next result point out the relation intertwining the splitting numbers of a closed geodesic $\gamma$ and the
splitting numbers of the  symplectic matrix (linearized Poincaré map) given in Definition. \ref{defn:S-S-N}.
(We refer the interested reader to \cite[pag. 191]{Lon02}).
 \begin{lem}\label{thm:R-S-N}
 Under the above notations, we have
 \begin{equation}\label{eq:R-S-N}
 \mathcal{S}^{\pm}(\omega,\gamma)=S^\pm_{\mathcal P(T)}(\omega), \forall \ \omega\in\U.
 \end{equation}
 \end{lem}
\begin{rem}
In \cite[pag.226]{BTZ82},  authors proved that the splitting numbers $\mathcal{S}^{\pm}(\omega,\gamma)$ only depend on the conjugacy class of
$P$ and in \cite[pag. 247]{Lon02}, Long  generalized this result. Here we provide a different proof with respect to
the one given by authors in \cite[pag.252]{Lon02}.
\end{rem}
\begin{proof}
For any  $\bar{\omega}\in \U$, we let $S\=\bar{\omega} A$. By Proposition \ref{Thm:R-O-M-M} and Remark \ref{Rem:C-C}, we have
\begin{equation}\label{eq:R-M-M}
\iMor(\bar{\omega},\gamma)+\nu(\bar{\omega} A)=\iCLM(\Gr(\bar \omega \trasp{A}_d),\Gr(\Phi(t)); t \in [0,T])
=\iCLM(\Gr(\bar \omega),\Gr(A_d\Phi(t)); t \in [0,T]).
\end{equation}
Moreover, by taking into account  Proposition \ref{prop:L-C}, we have 
\begin{equation}\label{eq:R-M-M-4}
\begin{aligned}
&\iCLM(\Gr(\bar \omega\Id),\Gr(A_d\Phi(t)); t \in [0,T])\\&=\iCLM(\Gr(\bar \omega\Id),\Gr(A_d\Phi(t))*\xi(t)); t \in [0,T])-\iCLM(\Gr(\bar \omega\Id),\Gr(\xi(t))); t \in [0,T])\\&=\begin{cases}
(\iota_{\bar \omega}(A_d\Phi(t)*\xi(t); t \in [0,T])+n)-(\iota_{\bar \omega}(\xi(t); t \in [0,T])+n) \quad \omega=1\\
\iota_{\bar \omega}(A_d\Phi(t)*\xi(t); t \in [0,T])-\iota_{\bar \omega}(\xi(t); t \in [0,T])\quad \omega\neq 1\\
\end{cases}
\\&=\iota_{\bar \omega}(A_d\Phi(t)*\xi(t); t \in [0,T])-\iota_{\bar \omega}(\xi(t); t \in [0,T])
\quad \forall \, \omega\in\U,
\end{aligned}
\end{equation} 
where $\xi$ is any symplectic path joining  $\Id$ to $A_d$. Summing up Equations \eqref{eq:R-M-M}-\eqref{eq:R-M-M-4},
we get
 \begin{equation}\label{eq:R-M-M-2}
\iMor(\bar{\omega},\gamma)=\iota_{\bar \omega}(A_d\Phi(t)*\xi(t))-\iota_{\bar \omega}(\xi(t))-\nu(\bar{\omega} A).
\end{equation}
According to Definition \ref{defn:S-S-N}, we infer that
\begin{equation}\label{eq:R-M-M-3}
\mathcal{S}^{\pm}(\bar \omega,\gamma)=S^\pm_{\mathcal P}(\bar \omega)-S^\pm_{A_d}(\bar \omega)+\nu(\bar{\omega} A)
\end{equation}
and being $A$  an orthogonal matrix, then by direct computation, we get
$S^\pm_{A_d}(\bar \omega)=\nu(\bar{\omega} A)$.
 This complete the proof.
\end{proof}
We let $\widetilde J\=-\sqrt{-1}J$. For any symplectic matrix $M\in\Sp(2n,\R)$, we define the $\widetilde J$-invariant (generalized eigenspace)
subspace $E_\lambda$ as
\begin{equation}
E_\lambda\=\bigcup_{m\geq1}\ker(M-\lambda I)^m.
\end{equation}
and we observe that the following $\widetilde J$-orthogonal splitting holds
\begin{equation}
\C^{2n}=\bigoplus_{\lambda\in \sigma(M)}E_{\lambda}.
\end{equation}
\begin{defn}\label{def:krein-type}
For any $\lambda\in\sigma(M)\cap\U$, the restriction of  $\widetilde J$ to the subspace $E_{\lambda}$ is non-degenerate.
We define the {\em Krein-type\/} of $\lambda$ by $(p,q)$,
where $p$ (resp. $q$) denotes respectively the total multiplicity of the positive (resp. negative) negative eigenvalues of $\widetilde J\mid_{E_{\lambda}}$.
 If $p=0$ (resp. $q=0$) the eigenvalue  $\lambda$ is termed {\em Krein-negative\/} (resp. {\em Krein-positive\/});
 otherwise, $\lambda$ is called {\em Krein-indefinite\/} or of {\em mixed-type.\/}
\end{defn}
For short, we will  refer to   the case $p=0$ or $q=0$ simply as {\em Krein-definite.\/}
Before recalling  the definition of {\em strongly stability\/}, we  introduce the following new definition.
\begin{defn}\label{def:trival geodesic}
A closed geodesic $\gamma$ is called {\em index hyperbolic\/} if, for any $\lambda\=e^{\sqrt{-1}2\pi \theta}\in\sigma(\mathcal P(T))\cap \U$ (eigenvalue of
the  Poincaré map $\mathcal P(T)=A_d\Phi(T)$, here $\Phi$ denotes the fundamental solution) it holds
\begin{multline}
S^+_{\mathcal P}(e^{\sqrt{-1}2\pi\theta}) =S^-_{\mathcal P}(e^{\sqrt{-1}2\pi\theta}) =0 \quad \textrm{ if }   \theta\in \Q,  \\
S^+_{\mathcal P}(e^{\sqrt{-1}2\pi\theta}) =S^-_{\mathcal P}(e^{\sqrt{-1}2\pi\theta}) \quad \textrm{ if  }  \theta\notin \Q.
\end{multline}
\end{defn}

\begin{prop}\label{thm:trivial}
Let $(M,\mathfrak g)$ be a Riemannian (resp. Lorentzian) closed  (resp. timelike and  closed) geodesic. We assume that
for any $m\in\N$, $\iMor\big(\gamma^{(m)}\big)=0$. Then  $\gamma$ is index hyperbolic.
\end{prop}
\begin{proof}
We start to observe that  by Proposition \ref{eq:I-F-C} we have
\begin{equation}\label{eq:above}
\iCLM\big(\Gr((\trasp{A}_d)^m),\Gr(\Phi(t));t\in[0,mT]\big)=\sum_{i=1}^{m}\iCLM\big(\Gr(\exp(\dfrac{i}{m}2\pi\sqrt{-1})\trasp{A}_d),\Gr(\Phi(t));t\in[0,T]\big).
\end{equation}
By invoking Proposition \eqref{eq:M-M},  Equation \eqref{eq:above} fits into the following
 \begin{equation}\label{eq:mancante}
\iMor(\gamma^{(m)})+\nu(A^m)=\sum_{i=1}^{m}\iMor\left(\exp(\dfrac{i}{m}2\pi\sqrt{-1}),\gamma\right)+\sum_{i=1}^{m}\nu(\exp(\dfrac{i}{m}2\pi\sqrt{-1}) A).
\end{equation}
By Equation \eqref{eq:nullity} we infer that $\nu(A^m)=\displaystyle{\sum_{i=1}^{m}}\nu(\exp(\dfrac{i}{m}2\pi\sqrt{-1}) A)$ and by substituting into Equation \eqref{eq:mancante},
we immediately get that
\begin{equation}
\iMor(\gamma^{(m)})=\sum_{i=1}^{m}\iMor\left(\exp(\dfrac{i}{m}2\pi\sqrt{-1}),\gamma\right).
\end{equation}
By assumption, for every $m \in \N$,   $\iMor(\gamma^{(m)})=0$. By definition,  $\iMor(\exp(\frac{i}{m}2\pi\sqrt{-1},\gamma)\geq0$ this immediately  implies that
\[
\iMor\left(\exp(\dfrac{i}{m}2\pi\sqrt{-1}),\gamma\right)=0 \quad \textrm{ for any } i, m\in\N
\]
which is  equivalent to assert that $\iMor(e^{\sqrt{-1}2\pi\theta},\gamma)=0$ for
any $\theta\in\Q$. So, for any $\lambda=e^{\sqrt{-1}2\pi\theta}\in\sigma(\mathcal P(T))$ and $\theta\in\Q$, by invoking Lemma \ref{thm:R-S-N}, we get
that
\[
S^\pm_{\mathcal P(T)}(\lambda)= \mathcal{S}^{\pm}(\lambda,\gamma)=0.
\]
This conclude the first claim. In order to prove the second claim,
if $\theta\notin\Q$, let $\theta_1<\theta<\theta_2$ be such that  $\theta_1,\theta_2$ are in $\Q$, $|\theta_j-\theta|$ is small enough and
$e^{\sqrt{-1}2\pi\theta_j}\notin\sigma(\mathcal P(T))$ for $j=1,2$. Then $\iMor(e^{\sqrt{-1}2\pi\theta_j},\gamma)=0$ and $\mathcal{S}^{\pm}(e^{\sqrt{-1}2\pi\theta_j},\gamma)=0$.  By Definition \ref{def:splitting-number-geo},  we know that in fact the splitting numbers $\mathcal{S}^{\pm}(\omega,\gamma)$ measure the jumps between $n_-(\omega,\gamma)$ and $n_-(\lambda,\gamma)$ 
 for $\lambda\in \U$ in a neighborhood of  $\omega$. By invoking once again Lemma \ref{thm:R-S-N} as well as Proposition \ref{prop:splitting number and ultimate type}, we infer that $\mathcal{S}^{\pm}(\omega,\gamma)=0$ if $\omega \notin \sigma(\mathcal P(T))$. In conclusion, we have
\begin{equation}
 \iMor(e^{\sqrt{-1}2\pi\theta_2},\gamma)=\iMor(e^{\sqrt{-1}2\pi\theta_1},\gamma)+\mathcal{S}^{+}(e^{\sqrt{-1}2\pi\theta_1},\gamma)+
 \mathcal{S}^{+}(\lambda,\gamma)-\mathcal{S}^{-}(\lambda,\gamma)-\mathcal{S}^{-}(e^{\sqrt{-1}2\pi\theta_2},\gamma),
\end{equation}
 so $\mathcal{S}^{+}(\lambda,\gamma)=\mathcal{S}^{-}(\lambda,\gamma)$. The conclusion, readily follows, by invoking once again
 Lemma \ref{thm:R-S-N}. This concludes the proof.
\end{proof}
We recall that the strong stability of a symplectic matrix can be characterized through its eigenvalues. For the
sake of the reader, we recall the following result proved by author in \cite{Eke90}.
\begin{lem}\label{lem:strong stability}({\bf  \cite[Thm 10, pag.11]{Eke90}\/})
$M$ is strongly stable if and only if it is linearly stable and all of its eigenvalues are Krein-definite.
\end{lem}
\begin{rem}\label{rem:strong stable}
We observe that since the eigenvalues $\{-1,1\}$ have  Krein-type $(p,p)$, by this it readily follows that  $\pm1$ cannot be in the spectrum  of a
strongly stable symplectic matrix.
\end{rem}
As a direct consequence of Proposition \ref{thm:trivial} and of the basic normal forms of a symplectic matrix, we get the
following strongly instability result for closed geodesics.
\paragraph{Proof of Theorem \ref{cor:non-strong stable}.}
Under  the assumptions of Theorem \ref{cor:non-strong stable}, by invoking Proposition \ref{thm:trivial},
we can conclude that the geodesic is  index hyperbolic. By taking into account Definition \ref{def:trival geodesic} on
index hyperbolicity applied to  the monodromy $\mathcal P(T)=A_d\Phi(T)$, we get that,
for every eigenvalue $e^{\sqrt{-1}\theta}\in \sigma(\mathcal P(T))$, $S^+_{\mathcal P(T)}(e^{\sqrt{-1}\theta}) -S^-_{\mathcal P(T)}(e^{\sqrt{-1}\theta}) =0 $.
In order to concludes the proof, we argue by contradiction. For, we assume that $\mathcal P(T)$ is strongly stable.
Thus, by the characterization given in Lemma \ref{lem:strong stability}, $\mathcal P(T)$ is linearly stable and  all of its
 eigenvalues  are Krein-definite. By this fact, it immediately  follows that, for any
eigenvalue $e^{\sqrt{-1}\theta}\in \sigma(\mathcal P(T))$ having Krein type $(p,q)$, it holds that $p-q\neq0$. By this we get a contradiction once invoked
\cite[Corollary 8, pag.198]{Lon02}, being
\begin{multline}
S^+_{\mathcal P}(e^{\sqrt{-1}\theta}) -S^-_{\mathcal P}(e^{\sqrt{-1}\theta})=p-q\neq 0 \textrm{ and } \\
S^+_{\mathcal P}(e^{\sqrt{-1}\theta}) -S^-_{\mathcal P}(e^{\sqrt{-1}\theta})=p-q= 0
\end{multline}
at the same time.  This concludes the proof.
\qed


\appendix

\section{On the Maslov  index, spectral flow and Index
theorems}\label{sec:appendix}

The aim of this Section is to  make a brief recap on the Maslov-type index,
the spectral flow for path of closed selfadjoint Fredholm operators and the
Index Theorems.

\subsection{On the Maslov-type index}\label{subsec:Maslov}

Our basic reference it is \cite{HS09, LZ00a, LZ00b, Lon02} and references
therein.

Let $ \omega\in \U$ and for any $ M\in \Sp(2n, \R)$, we
  define the real-valued function
  \[
D_{\omega}(M)=(-1)^{n-1}\overline\omega^{n}\det(M-\omega \Id_{2n}).
  \]
Then $ \Sp(2n, \R)^{0}_{\omega}\=\{M\in \Sp(2n, \R)| D_{\omega} M=0\} $ is a
codimensional-one
variety in $ \Sp(2n, \R) $ and let us define
\[
\Sp(2n)^{*}_{\omega}\=\Sp(2n)\backslash \Sp(2n, \R)^{0}_{\omega}=\Sp(2n,
\R)_{\omega}^{+}\cup
\Sp(2n, \R)_{\omega}^{-}
\]
where
\[
\Sp(2n,\R)_{\omega}^{+}\=\{M\in \Sp(2n,\R)| D_{\omega} M<0\} \textrm{ and }
\Sp(2n,\R)_{\omega}^{-}
\=\{M\in \Sp(2n,\R)
| D_{\omega} M>0\}.
\]
 For any $M\in \Sp(2n,\R)^{0}_{\omega}$, $\Sp(2n,\R)_{\omega}^{0}$ is
co-oriented at the point $M$
 by choosing  as positive direction the direction determined by
 $\frac{d}{dt}Me^{tJ}|_{t=0}$ with $t\geq0$ sufficiently small. The following
result is well-known.
 \begin{lem}\cite[pag.58-59]{Lon02}.\label{thm:topology}
 For any $\omega\in\U$, $\Sp(2n,\R)_{\omega}^{+}$ and
$\Sp(2n,\R)_{\omega}^{-}$ are two path connected
 components of $\Sp(2n,\R)_{\omega}^{*}$
  which are simple connected in $\Sp(2n,\R)$.
 \end{lem}
For any two $2m_{k}\times 2m_{k}$ matrices with the block form
$M_{k}=\begin{bmatrix} A_{k}&B_{k}\\C_{k}&D_{k}\end{bmatrix}$ with $k=1,2$, we
define the
 $\diamond$-product of $M_{1}$ and $M_{2}$ in the following way:
\[
M_{1}\diamond M_{2}\=\begin{bmatrix}
A_{1}&0&B_{1}&0\\0&A_{2}&0&B_{2}\\C_{1}&0&D_{1}&0\\0&C_{2}&0&D_{2}
\end{bmatrix}.
\]
The $k$-fold $\diamond$-product of $M$ is denoted by $M^{\diamond
k}=M\diamond\cdots \diamond M$.
 Note that the $\diamond$-product of two symplectic matrices is symplectic, so
for any two symplectic paths
 $T_{k}\in  \mathscr C^0\big([0,\tau],\Sp(2n_k, \R)\big),k=1,2$,
 denote $T_{1}\diamond T_{2}(t)=T_{1}(t)
 \diamond T_{2}(t),\forall t\in[0,\tau]$,
 then $ T_{1}\diamond T_{2}\in  \mathscr
C^0\big([0,\tau],\Sp(2(n_{k_{1}}+n_{k_{2}}), \R)\big)$.
Given any two symplectic paths $\gamma,\eta:[0,\tau]\rightarrow \Sp(2n,\R)$
such that $\gamma(0)=\eta(\tau)$, we define their concatenation in the
 following way:
\begin{equation}
  (\gamma*\eta)(t)\=
  \begin{cases}
   \eta(2t) & \textrm{ if }   0\leq t\leq \frac{\tau}{2}  \\
   \gamma(2t-\tau) & \textrm{ if  }  \frac{\tau}{2}\leq t\leq \tau.\\
   \end{cases}
  \end{equation}
For $a\in \R^*$, let
$D(a)=\begin{bmatrix}a&0\\0&\frac{1}{a}\end{bmatrix}$ and let
$M_{n}^{+}\=D(2)^{\diamond n}$ and $M_{n}^{-}\=D(-2)\diamond D(2)^{\diamond
(n-1)}$. By a straightforward computation  we have
$M_{n}^{+}\in \Sp(2n,\R)_{\omega}^{+}$
and $M_{n}^{-}\in\Sp(2n,\R)_\omega^{-}$. We let
\[
\eta_{n}(t)\=\begin{bmatrix}
2-\frac{t}{\tau}&0\\0&(2-\frac{t}{\tau})^{-1}\end{bmatrix}^{\diamond
n}\qquad t\in[0,\tau]
\]
and we observe that $\eta_{n}$ is a path in $\Sp(2n,\R)$
joining $M_{n}^{+}$ to $\Id_{2n}$. Following author in \cite{Lon02} we recall
the following definition.
 \begin{defn}\label{def:Maslov-index-ok}
For any $\omega\in \U$ and $ T \in  \mathscr C^0\big([0,\tau],\Sp(2n,
\R)\big)$ such that $T(0)=I_{2n}$, we define
\begin{equation}
\iota_\omega(T)\=\big[e^{-\varepsilon J}\big(T*\eta_{n}\big):\Sp(2n,\R)_{\omega}^{0}\big],
\end{equation}
where the (RHS) denotes the intersection number between the perturbed path
$t\mapsto e^{-\varepsilon J} (T*\eta_{n})(t)$
with the singular cycle $\Sp(2n,\R)^0_\omega$. We set
 \begin{equation}
\nu_{\omega}( T)=\dim_{\C}( T(\tau)-\omega \Id_{2n}).
\end{equation}
\end{defn}
\begin{rem}
 It is worth noticing that the Definition \ref{def:Maslov-index-ok} is
independent on the choice of a sufficiently
 small $\varepsilon>0$.
\end{rem}
By Definition \ref{def:Maslov-index-ok}, for any continuous symplectic path
$S:[a,b]\rightarrow \Sp(2n,\R)$ such that $S(a)\neq \Id_{2n}$, we can
choose a path $T \in  \mathscr C^0\big([a,b],\Sp(2n, \R)\big)$
such that $T(a)=\Id_{2n}$ and  $T(b)=S(a)$ and we define the{\em  Maslov-type index of $S$\/} as
\begin{equation}\label{eq:non-citata}
 \iota_\omega(S)\=\iota_{\omega}(S*T)-\iota_{\omega}(T).
\end{equation}
(For further details we refer the interested reader to
\cite[Definition 9, pag. 148 ]{Lon02}).
By \cite[Lemma 6,pag. 120 ]{Lon02}, for any
path $T \in  \mathscr C^0\big([0,\tau],\Sp(2n, \R)\big)$
such that $T(0)=\Id_{2n}$ the Maslov-type index
$\iota_\omega(T)$ is even (resp. odd) if and only if $e^{-\varepsilon
J}T(\tau)$ lies in $\Sp(2n,\R)_{\omega}^{+}$
(resp. $\Sp(2n,\R)_{\omega}^{-}$).
\begin{rem}
It is worth noticing that the  number $k$ appearing in \cite[Lemma
6,pag. 120]{Lon02} agrees with  $\iota_\omega(T)$.
We also observe that
$\beta(\tau)=M_{n}^{\pm}$ in particular implies  that $e^{-\varepsilon
J}T(\tau)\in\Sp(2n,\R)_{\omega}^{\pm}$.
\end{rem}
 \begin{lem}\label{thm:parity}
 Let $S:[a,b] \to \Sp(2n, \R)$ be a continuous path.   Then we have
\begin{itemize}
 \item[] $\iota_\omega(S)$ is even $\iff$ both the endpoints
 $e^{-\varepsilon J}S(a)$ and $e^{-\varepsilon J}S(b)$ lie in $\Sp(2n, \R)_\omega^+$ or
in $\Sp(2n, \R)_\omega^-$.
 \end{itemize}
 \end{lem}
\begin{proof}
As a direct application of  \cite[Lemma 6,pag. 120 ]{Lon02} to the paths $S*T$
and $T$, we immediately
get that, if $e^{-\varepsilon J}S(a)$ and $e^{-\varepsilon J}S(b)$ are in the
same
components of $\Sp(2n,\R)_{\omega}^{*}$, then  the parity of the two Maslov-type 
indices $\iota_\omega(S*T)$ and $\iota_\omega(T)$ coincides. Furthermore,
by the definition stated in Formula \eqref{eq:non-citata},
we know that it's  even and hence also the converse is true. This conclude the
proof.
\end{proof}

To use the iteration formula, here we give the definition and some properties we need of the splitting numbers of any symplectic matrix which can be found in \cite[pag.190-199]{Lon02}.
\begin{defn}\label{defn:S-S-N}
For any $M\in\Sp(2n)$ and $\omega\in\U$, we define the splitting numbers $S_{M}^{\pm}(\omega)$ of $M$ by
$
S_{M}^{\pm}(\omega)=\lim\limits_{\theta\rightarrow0^{\pm}}\iota_{\omega e^{\sqrt{-1}\theta}}(\gamma)-\iota_\omega(\gamma)
$
 where $\gamma$ is a symplectic path connecting $I_{2n}$ and $M$.
\end{defn}

The splitting numbers have the following property:
\begin{prop}\label{prop:propery-splitting}
The splitting numbers $S_{M}^{\pm}(\omega)$ are independent of the path $\gamma$ and for any $\omega\in\U$ and $N\in\Omega^0(M)$, the path connected component of $\Omega(M)=\{N\in\Sp(2n)\mid\sigma(N)\cap\U=\sigma(M)\cap\U,
\text{and} \ \nu_\lambda(N)=\nu_\lambda(M)\ \forall \lambda\in\sigma(M)\cap\U\}$ which contains $M$, $S_{N}^{\pm}(\omega)$ are
constant. If $\omega\notin\sigma(M)$, then $S_{M}^{\pm}(\omega)=0$. Moreover,
\begin{equation}
S_{M_1\diamond M_2}^{\pm}(\omega)=S_{M_1}^{\pm}(\omega)+S_{M_2}^{\pm}(\omega), \forall \ \omega\in \U.
\end{equation}
\end{prop}

For any symplectic matrix $M\in\Sp(2n,\R)$ with eigenvalue $\omega\in\mathbf{U}$, in order to give a complete explanation of $S^{\pm}_M(\omega)$,
we need the concept of ultimate type of $\omega$ which is introduced in \cite[pag.41-42]{Lon02}.

Firstly, we give all the basic normal forms for eigenvalues of $M$ in $\mathbf{U}$ as follows:
\begin{equation}
N_1(\lambda,b)=\begin{bmatrix}\lambda&b\\0&\lambda\end{bmatrix}, \ \lambda=\pm1,b=\pm1,0
\end{equation}
\begin{equation}
R(\theta)=\begin{bmatrix}\cos\theta&-\sin\theta\\ \sin\theta&\cos\theta \end{bmatrix}, \ \theta\in(0,\pi)\cup(\pi,2\pi)
\end{equation}
\begin{equation}
N_2(\omega,b)=\begin{bmatrix}R(\theta)&b\\ 0&R(\theta)\end{bmatrix},\ \theta\in(0,\pi)\cup(\pi,2\pi)
\end{equation}
where $b=\begin{bmatrix}b_1&b_2\\ b_3&b_4\end{bmatrix}$ such that $b_i\in \mathbf{R}$ and $b_2\neq b_3$.
A basic normal form $N\in\Sp(2n)$ is called trivial if $NR((t-1)\alpha)^{\diamond n}$ possesses no eigenvalues on $\mathbf{U}$
for sufficiently small $\alpha>0$ and $t\in[0,1)$, otherwise it is called non-trivial. Then the ultimate type $(p,q)$ of
$\omega\in\sigma(N)\cap\mathbf{U}$ is defined to be its Krein-type $(p,q)$ if $N$ is non-trivial, and to be $(0,0)$ if $N$ is trivial.
The definition of Krein-type is referred to Definition \ref{def:krein-type}. Moreover, if $\omega \in\mathbf{U}\setminus\sigma(N)$, then its ultimate
type is defined to be $(0,0)$. Note that for any $M\in\Sp(2n)$, there is a path $f:[0,T]\rightarrow\Omega^0(M)$ such that $f(0)=M$ and
$f(1)=M_1(\omega_1)\diamond \cdots \diamond M_k(\omega_k)\diamond M_0$, where $M_i(\omega_i)$ is a basic normal form of some eigenvalue
$\omega\in \mathbf{U}$ for $1\leq i\leq k$, and the eigenvalues of $M_0$ are not on $\mathbf{U}$ (where $\Omega^0(M)$ was defined in Proposition \ref{prop:propery-splitting}). Now we can give the definition of the ultimate
type of any $M\in\Sp(2n,\R)$.
\begin{defn}\label{def:ultimate type}
Under above notations, the ultimate type of $\omega$ for $M$ is defined to be $(p,q)$ by
\begin{equation}
p=\sum_{i=1}^kp_i,   \quad q=\sum_{i=1}^kq_i
\end{equation}
where $(p_i,q_i)$ is the ultimate type of $\omega$ for $M_i$.
\end{defn}

The relationship between the splitting numbers and the ultimate type of $\omega$ for $M$ is given by the following proposition in
 \cite[Thm 7,pag.192]{Lon02}:
\begin{prop}\label{prop:splitting number and ultimate type}
For any $\omega\in\mathbf{U}$ and $M\in\Sp(2n,\R)$,
\begin{equation}
S^+_M(\omega)=p, \quad S^-_M(\omega)=q
\end{equation}
where $(p,q)$ is the ultimate type of $\omega$ for $M$.
\end{prop}

We close this section with a technical useful result which will be used in the proof of the main instability criterion and
was proved in
\cite[lemma 3.2]{HS10}.
\begin{prop}\label{thm:lemma3.2}
 Let $T:[0,\tau] \to \Sp(2n, \R)$ be a continuous symplectic path such that $T(0)$ is linearly stable.
 \begin{enumerate}
  \item  If $1 \notin \sigma\big(T(0)\big)$ then there exists $\varepsilon >0$
sufficiently small such that
  $T(s) \in \Sp(2n, \R)^+$ for $|s| \in (0, \varepsilon)$.
  \item We assume that $\dim \ker \big(T(0)-\Id_{2n}\big)=m$ and
  $\trasp{T(0)}J T'(0)\vert_V$ is non-singular for $V\= T^{-1}(0) \R^{2m}$. If
  $\ind\big({\trasp{T(0)}J T'(0)\vert_V}\big)$ is even [resp. odd] then there
exists $\delta>0$
  sufficiently small such that  $T(s) \in \Sp(2n, \R)^+$
  [resp. $T(s) \in \Sp(2n, \R)^-$] for $|s| \in (0, \varepsilon)$.
  \end{enumerate}
\end{prop}

\subsubsection{On the Maslov index after Cappell-Lee-Miller }
In this subsection, we will give the definition of the Maslov index following
Cappell-Lee-Miller.

Let $J=\begin{bmatrix} 0&-\Id_n\\\Id_n&0 \end{bmatrix}$,
 then $(\R^{2n},\omega)$ can be seen as a symplectic vector space with the
symplectic form $\omega$ such that $\omega(x,y)=\langle Jx,y\rangle$ for any
 $x,y\in \R^{2n}$. A subspace $\L$ is Lagrangian if and only if $\omega\vert
_{\L}=0 $ and $\dim\L=n$. Let us consider the Lagrangian Grassmannian of
$(\R^{2n},\omega)$, namely the set
$\Lagr(\R^{2n}, \omega)$ of all Lagrangian subspaces. Recall that it is a real
compact and connected
analytic embedded $n(n+1)/2$-dimensional submanifold of the Grassmannian manifold of $\R^{2n}$.
Given  $L_0\in\Lagr(\R^{2n},\omega)$ and any non-negative integer $j \in
\Set{0,\dots, n}$, we define the sets
$\Lagr^j(L_0;\R^{2n})\=\big\{L\in\Lagr(\R^{2n},\omega):\dim(L\cap L_0)=j\big\}$
and we observe that $\Lagr(\R^{2n},\omega)\= \bigcup_{j=0}^n \Lagr^j(L_0;\R^{2n})$. It is well-known
that
$\Lagr^j(L_0;\R^{2n})$ is a connected embedded analytic  submanifold of
$\Lagr(\R^{2n},\omega)$  having
codimension equal to $ j(j+1)/2$. In particular $\Lagr^1(L_0;\R^{2n})$
has codimension $1$ and for  $ j \geq 2$ the codimension of
$\Lagr^j(L_0;\R^{2n})$ in $\Lagr(\R^{2n},\omega)$
is bigger or equal to $3$.  We define the Maslov (singular) cycle with vertex at
$L_0$ as follows:
\[
\Sigma(L_0;\R^{2n}) \= \bigcup_{j=1}^n \Lagr^j(L_0;\R^{2n}).
\]
We note that the Maslov  cycle is the closure of the lowest
codimensional stratum
$\overline{\Lagr^1(L_0;\R^{2n})}$. In particular, $\Lagr^0(L_0;\R^{2n})$, the
set of all Lagrangian
subspaces that are transversal to $L_0$, is an open and  dense subset of
$\Lagr(\R^{2n},\omega)$.
The (top stratum) codimensional 1-submanifold $\Lagr^1(L_0;\R^{2n})$ in
$\Lagr(\R^{2n},\omega)$ is co-oriented or
otherwise stated it carries a transverse orientation. In fact
given $\varepsilon >0$, for each $L \in \Lagr^1(L_0;\R^{2n})$, the smooth
path of Lagrangian subspaces $\ell:(-\varepsilon, \varepsilon) \to
\Lagr(\R^{2n},\omega)$ defined by
$\ell(t)\=\exp(tJ)$ crosses $\Lagr^1(L_0;\R^{2n})$ transversally. The desired
transverse orientation is given by
the direction along the path when the parameter runs between $(-\varepsilon,
\varepsilon)$.
Thus the Maslov cycle is two-sidedly embedded in $\Lagr(\R^{2n},\omega)$. Based
on
these properties, Arnol'd in \cite{Arn67}, defined an intersection index for
closed loops in  $(\R^{2n}, \omega)$ via
transversality arguments. Following authors in \cite{CLM94, HS09} we introduce
the following Definition.
\begin{defn}\label{def:Maslov-index}
Let $L_0 \in \Lagr(\R^{2n},\omega)$ and, for $a<b$, let  $\ell\in \mathscr
C^0\big([a,b],\Lagr(\R^{2n},\omega)\big)$. We define
the Maslov index of $\ell$ with respect to $L_0$ as the integer given by
\begin{equation}\label{eq:intersection-number}
 \iCLM(L_0, \ell)\=\left[\exp(-\varepsilon J)\,\ell: \Sigma(L_0;\R^{2n})\right]
\end{equation}
where $\varepsilon \in (0,1)$ is sufficiently small and where the right-hand
side denotes the intersection number.
\end{defn}
\begin{rem}
A few Remarks on the Definition  \ref{def:Maslov-index} are in order. By the
basic geometric
observation given in \cite[Lemma 2.1]{CLM94}, it readily follows that there
exists $\varepsilon >0$
sufficiently small such that $\exp(-\varepsilon J)\,\ell(a), \exp(-\varepsilon
J)\,\ell(b)$ doesn't lie
on $\Sigma(L_0;\R^{2n})$. By \cite[Step 2, Proof of Theorem 2.3]{RS93}, there
exists a perturbed path $\widetilde \ell$
having only simple crossings(namely the path $\ell$ intersects the Maslov cycle
transversally and  in
the top stratum). Since, simple crossings are isolated, on a compact interval
are in a finite number. To each
crossing instant  $t_i \in (a,b)$ we associate the number $s(t_i)= 1$ (resp.
$s(t_i)=-1$) according to the fact that,
in a sufficiently small neighbourhood of $t_i$, $\widetilde \ell$ have the same
(resp. opposite) direction of
$\exp(t\,J) \widetilde \ell(t_i)$. Then the intersection number given in Formula
\eqref{eq:intersection-number} is
equal to the summation of $s(t_i)$, where the sum runs over all crossing
instants $s(t_i)$.
\end{rem}
\noindent

The Maslov index given in Definition \ref{def:Maslov-index} have many important
properties (cfr. \cite{RS93, CLM94}
for further dails).
\begin{enumerate}
 \item[]{\bf Property I (Reparametrisation Invariance)\/}
 Let $\psi:[a,b] \to [c,d]$ be a  continuous function with $\psi(a)=c$ and
$\psi(b)= d$.
 Then $ \iCLM(L_0, \ell)=\iCLM(L_0, \ell\circ \psi).$
 \item[]{\bf Property II (Homotopy invariance Relative to the Ends)\/} Let
\[
\bar  \ell: [0,T] \times [a,b] \to
\Lagr(V,\omega):(s,t)\mapsto \bar\ell(s,t)
\]
be a continuous two-parameter family of Lagrangian subspaces such that
$\dim\big(L_0\cap \bar \ell(s,a)\big) $
and $\dim\big(L_0\cap \bar{\ell(s,b)}\big) $ are independent on $s$. Then $
 \iCLM(L_0, \bar\ell_0) = \iCLM(L_0, \bar\ell_1)$
where $\bar \ell_0(\cdot)\= \bar\ell(0, \cdot)$ and
$\bar \ell_1(\cdot)\= \bar\ell(1, \cdot)$.
\item[]{\bf Property III (Path Additivity)\/} If $c \in (a,b)$, then
\[
\iCLM(L_0, \ell)= \iCLM(L_0, \bar\ell\vert_{[a,c]})+ \iCLM(L_0,
\bar\ell\vert_{[c,b]}).
\]
\item[]{\bf Property IV (Symplectic Invariance)\/} Let $\phi \in \mathscr
C^0\big([a,b], \Sp(V,\omega)\big)$ be a
 continuous path in the (closed) symplectic group $\Sp(V,\omega)$ of all
symplectomorphisms of $(V,\omega)$. Then
 \[
  \iCLM(L_0, \ell)= \iCLM\big(\phi(t)\, L_0, \phi(t)\, \ell(t)\big), \qquad t
\in [0,T].
 \]
 \item[]{\bf Property V (Symplectic Additivity)\/} For $i=1,2$ let
 $(V_i, \omega_i)$  be symplectic  vector spaces, $L_i \in \Lagr(V_i,\omega_i)$
and let $\ell_i \in
 \mathscr C^0\big([a,b], \Lagr(V_i,\omega_i\big)$. Then
 \[
\iCLM( L_1\oplus L_2,\ell_1\oplus \ell_2)= \iCLM( L_1,\ell_1)+ \iCLM(
L_2,\ell_2).
 \]
 \end{enumerate}

One  efficient technique for computing this invariant, was introduced (in the
non-degenerate case)
by the authors in \cite{RS93} through  the so-called crossing forms and,
generalised (in the degenerate
situation) by authors in \cite{GPP03, GPP04}.
For $\varepsilon >0$ let $\ell^* :(-\varepsilon, \varepsilon) \to
\Lagr(\R^{2n},\omega)$ be a $\mathscr C^1$-path such that
$\ell^*(0)=L$. Let $L_1$ be a fixed Lagrangian complement of
$L$ and, for $v \in L$ and for sufficiently small $t$ we define $w(t) \in L_1$
such that
$v+w(t) \in \ell^*(t)$. Then the  form
\begin{equation}\label{eq:laq}
 Q[v]=\dfrac{d}{dt}\Big|_{t=0} \omega\big(v, w(t)\big)
\end{equation}
is independent of the choice of $L_1$. A {\em crossing instant\/} $t_0$ for the
continuous curve $\ell:[a,b] \to \Lagr(\R^{2n},\omega)$ is an instant such that
$\ell(t_0)\in \Sigma(L_0;\R^{2n})$. If the
curve is $\mathscr C^1$, at each
crossing, we define the crossing form as  the quadratic form on  $\ell(t_0)\cap
L_0$ given by
\[
\Gamma(\ell, L_0, t_0)= Q\big(\ell(t_0), \dot \ell(t_0)\big)\Big\vert_{\ell(t_0)\cap L_0}
\]
where $Q$ was defined in Formula \eqref{eq:laq}. A crossing $t_0$ is called
regular if the crossing
form is non-degenerate; moreover if the curve $\ell$ has only
regular crossings we shall refer as a regular path.  (Heuristically, $\ell$ has
only regular
crossings if and only if it is transverse to $\Sigma(L_0)$).  Following authors
in \cite{LZ00b}, if
$\ell: [a,b] \to \Lagr(\R^{2n},\omega)$ is a regular $\mathscr C^1$-path,
then the crossing instants are in a finite number and the Maslov index is given
by:
\[
 \iCLM(L_0,\ell)=\coindex\left[{\Gamma(\ell(a), L_0, a)}\right]+
 \sum_{\substack{t \in(a,b)}}\ssgn{\Gamma(\ell(t), L_0, t)} -
 \iMor{\left[\Gamma(\ell, L_0, b)\right]},
\]
where $\coindex, \iMor $ denotes respectively the number of positive (coindex),
negative eigenvalues (
index) in the  Sylvester's Inertia Theorem and where $\sgn\=\coindex-\iMor$
denotes the (signature). We
observe that any $\mathscr C^1$-path is homotopic through a fixed endpoints
homotopy to a path
having  only regular  crossings.


\subsection{On the spectral flow}\label{subsec:spectral-flow}

The aim of this subsection is to briefly recall the Definition and the main
properties of the spectral
flow for a continuous path of closed selfadjoint Fredholm operator.
Our basic reference is \cite{Wat15}  and references therein.

Let $\mathcal H$ be a separable complex Hilbert space and let
$A: \mathcal D(A) \subset \mathcal H \to \mathcal H$ be  a  selfadjoint
Fredholm
operator. By the Spectral decomposition Theorem (cf., for instance,
\cite[Chapter III,
Theorem 6.17]{Kat80}), there is an orthogonal decomposition $
 \mathcal H= E_-(A)\oplus E_0(A) \oplus E_+(A),$
that reduces the operator $A$
and has the property that
\[
 \sigma(A) \cap (-\infty,0)=\sigma\big(A_{E_-(A)}\big), \quad
 \sigma(A) \cap \{0\}=\sigma\big(A_{E_0(A)}\big),\quad
 \sigma(A) \cap (0,+\infty)=\sigma\big(A_{E_+(A)}\big).
\]
\begin{defn}\label{def:essential}
Let $A \in \mathcal{CF}^{sa}(\mathcal H)$. We  term $A$ {\em essentially
positive\/}
if $\sigma_{ess}(A)\subset (0,+\infty)$, {\em essentially negative\/} if
$\sigma_{ess}(A)\subset (-\infty,0)$ and finally
{\em strongly indefinite\/} respectively if $\sigma_{ess}(A) \cap (-\infty,
0)\not=
\emptyset$ and $\sigma_{ess}(A) \cap ( 0,+\infty)\not=\emptyset$.
\end{defn}
\noindent
If $\dim E_-(A)<\infty$,
we define its {\em Morse index\/}
as the integer denoted by $\iindex{A}$ and defined as $
 \iindex{A} \= \dim E_-(A).$
Given $A \in\cfsa(\mathcal H)$, for  $a,b \notin
\sigma(A)$ we set
\[
\mathcal P_{[a,b]}(A)\=\Real\left(\dfrac{1}{2\pi\, i}\int_\gamma
(\lambda-A)^{-1} d\, \lambda\right)
\]
where $\gamma$ is the circle of radius $\frac{b-a}{2}$ around the point
$\frac{a+b}{2}$. We recall that if
$[a,b]\cap \sigma(A)$ consists of  isolated eigenvalues of finite type then
$
 \im \mathcal P_{[a,b]}(A)= E_{[a,b]}(A)\= \bigoplus_{\lambda \in (a,b)}\ker
(\lambda -A);
$
(cf. \cite[Section XV.2]{GGK90}, for instance) and $0$ either belongs in the
resolvent set of $A$ or it is an isolated eigenvalue of finite multiplicity.
Let us now consider the {\em graph distance topology\/} which is the topology
induced by the {\em gap
metric\/} $d_G(A_1, A_2)\=\norm{P_1-P_2}$
where $P_i$ is the projection onto the graph of $A_i$ in the product space
$\mathcal H
\times \mathcal H$. The next result allow us to  define the spectral flow for
gap
continuous paths in  $\cfsa(\mathcal H)$.
\begin{prop}\label{thm:cor2.3}
 Let $A_0 \in \cfsa(\mathcal H)$ be fixed.
 \begin{enumerate}
  \item[(i)] There exists a positive real number $a \notin \sigma(A_0)$ and an
open neighborhood $\mathscr N \subset  \cfsa(\mathcal H)$ of $A_0$ in the gap
topology such that $\pm a \notin
\sigma(A)$ for all $A \in  \mathscr N$ and the map
 \[
  \mathscr N \ni A \longmapsto \mathcal P_{[-a,a]}(A) \in \Lin(\mathcal H)
 \]
is continuous and the projection $\mathcal P_{[-a,a]}(A)$ has constant finite
rank for all $A \in \mathscr N$.
 \item[(ii)] If $\mathscr N$ is a neighborhood as in (i) and $-a \leq c \leq d
\leq a$ are such that $c,d \notin
 \sigma(A)$ for all $A \in \mathscr N$, then $A \mapsto \mathcal P_{[c,d]}(A)$
is
continuous on $\mathscr N$.
 Moreover the rank of $\mathcal P_{[c,d]}(A) \in \mathscr N$ is finite and
constant.
 \end{enumerate}
\end{prop}
\begin{proof}
For the proof of this result we refer the interested reader to
\cite[Proposition 2.10]{BLP05}.
\end{proof}
Let $\mathcal A:[c,d] \to \cfsa(\mathcal H)$ be a gap continuous path.  As
consequence
of Proposition \ref{thm:cor2.3}, for every $t \in [c,d]$ there exists $a>0$ and
an open
connected neighborhood $\mathscr N_{t,a} \subset \cfsa(\mathcal H)$ of
$\mathcal
A(t)$
such that $\pm a \notin \sigma(A)$ for all $A\in \mathscr N_{t,a}$ and the map
$\mathscr N_{t,a} \in A \longmapsto \mathcal P_{[-a,a]}(A) \in \mathcal
B$
is continuous and hence $ \rk\left(\mathcal P_{[-a,a]}(A)\right)$ does not
depends on $A \in \mathscr N_{t,a}$. Let us consider the open covering
of the interval $[c,d]$ given by the
pre-images of the neighborhoods $\mathcal
N_{t,a}$ through $\mathcal A$ and, by choosing a sufficiently fine partition of
the interval $[a,b]$ having diameter less than the Lebesgue number
of the covering, we can find  $c=:t_0 < t_1 < \dots < t_n:=d$,
operators $T_i \in \cfsa(\mathcal H)$ and
positive real numbers $a_i $, $i=1, \dots , n$ in such a way the restriction of
the path $\mathcal A$ on the
interval $[t_{i-1}, t_i]$ lies in the neighborhood $\mathscr N_{t_i, a_i}$ and
hence the
$\dim E_{[-a_i, a_i]}(\mathcal A_t)$ is constant for $t \in [t_{i-1},t_i]$,
$i=1, \dots,n$.
\begin{defn}\label{def:spectral-flow-unb}
The \emph{spectral flow of $\mathcal A$} (on the interval $[c,d]$) is defined by
\[
\spfl(\mathcal A, [c,d])\=\sum_{i=1}^n \dim\,E_{[0,a_i]}(\mathcal A_{t_i})-
 \dim\,E_{[0,a_i]}(\mathcal A_{t_{i-1}}) \in \Z.
\]
\end{defn}
(In shorthand Notation we  denote  $\spfl(\mathcal A, [a,b])$ simply  by
$\spfl(\mathcal A)$ if no confusion  is possible).
The spectral flow as given in Definition \ref{def:spectral-flow-unb} is
well-defined
(in the sense that it is independent either on the partition or on the $a_i$)
and only depends on
the continuous path $\mathcal A$. (Cfr. \cite[Proposition 2.13]{BLP05} and
references therein).
Here we list one of the useful properties of the spectral flow and we refer to \cite{BLP05}
for further details.
\begin{itemize}
 \item[]  {\bf (Path Additivity)\/} If $\mathcal A_1,\mathcal
A_2: [a,b] \to
 \cfsa(\mathcal H)$ are two continuous path such that
$\mathcal A_1(b)=\mathcal A_2(a)$, then
 $
  \spfl(\mathcal A_1 *\mathcal A_2) = \spfl(\mathcal A_1)+\spfl(\mathcal A_2).
 $
\end{itemize}
As already observed, the spectral flow, in general,  depends on the whole path
and not
just on the ends. However, if the path has a special form, it actually depends
on the
end-points. More precisely, let  $\mathcal A ,\mathcal B\in \cfsa(\mathcal H)$
and let $\widetilde{\mathcal A}:[a,b] \to \cfsa(\mathcal H)$ be the path
pointwise defined by $\widetilde{\mathcal A}(t)\=\mathcal A+ \widetilde{\mathcal
B}(t)$  where $
\widetilde{\mathcal B}$ is any continuous curve of $\mathcal A$-compact
operators parametrised on $[a,b]$
such that  $\widetilde{\mathcal B}(a)\=0$ and  $ \widetilde{\mathcal B}(b)\=
\mathcal B$. In this case,
the spectral flow depends of the
path $\widetilde A$, only on the endpoints (cfr. \cite{LZ00a} and reference
therein).
\begin{rem}
 It is worth noticing that, since every operator $\widetilde{\mathcal A}(t)$ is
a compact perturbation of a
 a fixed one, the path $\widetilde{\mathcal A}$ is actually a continuous path
into $\Lin(\mathcal W; \mathcal H)$,
 where $\mathcal W\=\mathcal D(\mathcal A)$.
\end{rem}
\begin{defn}\label{def:rel-morse-index}(\cite[Definition 2.8]{LZ00a}).
 Let $\mathcal A ,\mathcal B\in \cfsa(\mathcal H)$ and we assume that $\mathcal
B$ is
 $\mathcal A$-compact (in the sense specified above). Then the
{\em  relative Morse index of the pair $\mathcal A$, $\mathcal A+\mathcal B$\/}
is
defined by $
  \irel(\mathcal A, \mathcal A+\mathcal B)=-\spfl(\widetilde{\mathcal A};[a,b])$
where $\widetilde{\mathcal A}\=\mathcal A+ \widetilde{\mathcal B}(t)$ and where
$
\widetilde{\mathcal B}$ is any continuous curve parametrised on $[a,b]$
of $\mathcal A$-compact operators such that
$\widetilde{\mathcal B}(a)\=0$ and
$ \widetilde{\mathcal B}(b)\= \mathcal B$.
\end{defn}
\noindent
In the special case in which the Morse index of both operators $\mathcal A$ and
$\mathcal A+\mathcal B$ are
finite, then
\begin{equation}\label{eq:miserve}
\irel(\mathcal A, \mathcal A+\mathcal B)=\iindex{\mathcal A +\mathcal
B}-\iindex{\mathcal A}.
\end{equation}

Let  $\mathcal W, \mathcal H$ be separable Hilbert spaces with a dense and
continuous
inclusion $\mathcal W \hookrightarrow \mathcal H$ and let
$\mathcal A:[a,b] \to \cfsa(\mathcal H)$  having fixed domain $\mathcal W$. We
assume that $\mathcal A$ is
a continuously differentiable path  $\mathcal A: [a,b] \to \cfsa(\mathcal H)$
and
we denote by $\dot{\mathcal A}_{\lambda_0}$ the derivative of
$\mathcal A_\lambda$ with respect to the parameter $\lambda \in [a,b]$ at
$\lambda_0$.
\begin{defn}\label{def:crossing-point}
 An instant $\lambda_0 \in [a,b]$ is called a {\em crossing instant\/} if
$\ker\, \mathcal A_{\lambda_0} \neq 0$. The
 crossing form at $\lambda_0$ is the quadratic form defined by
\begin{equation}
 \Gamma(\mathcal A, \lambda_0): \ker \mathcal A_{\lambda_0} \to \R, \quad
\Gamma(\mathcal A, \lambda_0)[u] =
\langle \dot{\mathcal A}_{\lambda_0}\, u, u\rangle_\mathcal H.
\end{equation}
Moreover a  crossing $\lambda_0$ is called {\em regular\/}, if $\Gamma(\mathcal
A, \lambda_0)$ is non-degenerate.
\end{defn}
We recall that there exists $\varepsilon >0$ such that   $\mathcal A +\delta \,
\Id_\mathcal H$ has only regular crossings
  for almost every $\delta \in (-\varepsilon, \varepsilon)$. (Cfr., for instance
\cite[Theorem 2.6]{Wat15}
  and references therein).
In the special case in which all crossings are regular, then the spectral flow
can be easily computed through  the
crossing forms. More precisely the following result  holds.
\begin{prop}\label{thm:spectral-flow-formula}
 If $\mathcal A:[a,b] \to \cfsa(\mathcal W, \mathcal H)$ has only regular
crossings then they are in a finite
 number and
 \[
  \spfl(\mathcal A, [a,b]) = -\iMor{\left[\Gamma(\mathcal A,a)\right]}+
\sum_{t_0 \in (a,b)}
  \sgn\left[\Gamma(\mathcal A, t_0)\right]+
  \coiindex{\Gamma(\mathcal A,b)}
 \]
where the sum runs over all the crossing instants.
\end{prop}
\begin{proof}
 The proof of this result follows by arguing as in \cite{RS95}. 
\end{proof}

\subsection{Index Theorem for Hamiltonian Systems}\label{subsec:Index-Theorem}

Let  $H \in \mathscr C^2\Big([0,T] \times \R^{2n},\R\Big)$ be a  time-dependent
Hamiltonian function
and let $L$ be a Lagrangian subspace of the symplectic space $(\R^{2n}\oplus
\R^{2n}, -\omega \oplus \omega)$.
We define the closed (in $L^2$) subspace
$\mathcal D(T, L)\= \Set{z\in W^{1,2}([0,T], \R^{2n})| \big(z(0), z(T)\big) \in
L}$ and
denoting by $\bar{\mathcal D(T,L)}$ the closure in the $W^{1/2,2}$-norm
topology of
$\mathcal D(T,L)$, let us consider the {\em symplectic action functional\/}
$\mathscr A_H: \bar{\mathcal D(T,L)}\to \R$ defined by
\begin{equation}\label{eq:symplectic-action}
 \mathbb A_H(z)\=\int_0^T \left[\Big\langle -J\, \dfrac{dz(t)}{dt},
z(t)\Big\rangle- H\big(t, z(t)\big)\right]\, dt.
\end{equation}
By standard computation it follows that a critical point of $\mathbb A_H$ is a
weak (in the Sobolev
sense)-solution of the  following boundary value problem
\begin{equation}\label{eq:ham-sys-1}
\begin{cases}
 \dot z(t)= \nabla H\big(t, z(t) \big),  \qquad t \in [0,T]\\
 \big(z(0), z(T)\big) \in L.
\end{cases}
\end{equation}
\begin{rem}
We observe that the periodic solutions can be obtained by setting $L=\Delta$
where $\Delta$ denotes the diagonal subspace
in the product space $\R^{2n} \oplus \R^{2n}$.
\end{rem}
Let $z$ be a solution of the Hamiltonian System given in Equation
\eqref{eq:ham-sys-1} and let us denote by $\gamma$ the fundamental solution of
its linearisation along $z$; namely
$\gamma: [0,T] \to \Sp(2n)$ is the solution of the following Cauchy problem
\begin{equation}\label{eq:ham-sys-1-lin}
\begin{cases}
 \dot\gamma(t)= D^2H\big(t, z(t)\big)\, \gamma(t),  \qquad t \in [0,T]\\
 \gamma(0)=\Id_{2n}
\end{cases}.
\end{equation}
We set  $B(t)\=D^2H\big(t, z(t)\big)$ and we set $A_1\=-
J\dfrac{d}{dt} - B(t)$ and
$ A_0\=-J\dfrac{d}{dt}$ be the
closed selfadjoint Fredholm operators in $L^2$ with domain
\[
\mathcal D(T, L)\= \Set{z\in W^{1,2}([0,T], \R^{2m})| \big(z(0), z(T)\big) \in
L}.
\]
We define the {\em relative Morse
index\/} of $z$ as follows
\begin{equation}\label{eq:relative-Morse-index}
 \iRel(z)\= -\spfl\left(
A;[0,T]\right)
\end{equation}
where $ A:[0,1] \to \cfsa(\mathcal H)$ is a continuous path of closed
selfadjoint Fredholm operators defined
by  $ A(s)\= A_0 +  B(s)$ where the continuous path
$s\mapsto B(s)$ is such that $B(0)= 0$ and $
B(1)\= B$
on the $s$-independent domain $\mathcal D(T,L)$. We define the {\em Maslov
index\/} of the solution $z$ as follows
\begin{equation}\label{eq:Maslov-z}
 \iMas(z)\= \iCLM\big(L, \Graph(\gamma);[0,T]\big).
\end{equation}
We observe that $z_s \in \ker \left(\mathcal A(s)|_{\mathcal D(T, L)}\right)$ if
and only if $z_s$ is  a solution of the
linear Hamiltonian boundary value problem
\begin{equation}\label{eq:ham-sys-family}
\begin{cases}
 \dot z_s= J\, B_s(t)\, z_s(t),  \qquad t \in [0,T]\\
 \big(z_s(0), z_s(T)\big) \in L \cap \Graph\big(\gamma_s(T)\big)
\end{cases}
\end{equation}
where $\gamma_s$ is the fundamental solution of the Equation in
\eqref{eq:ham-sys-family}.
\begin{prop}({\bf A Morse-type Index Theorem\/}) Under the
above Notation we have
\[
 \iMas(z)= \iRel(z).
\]
\end{prop}
\begin{proof}
For the proof of this result we refer the interested reader to \cite[Theorem
2.5]{HS09}.
\end{proof}
\begin{rem}
 It is worth noticing that if $L=L_1 \oplus L_2 \in \Lagr(\R^{2m}\oplus \R^{2m},
-\omega \oplus \omega)$,
 where $L_i \in \Lagr(\R^{2m},\omega)$, for $i=1,2$, then we  have
 $\iCLM\big(L_1 \oplus L_2, \Graph(\gamma); [0,T]\big)= \iCLM\big(L_2,
\ell_1;[0,T]\big)$ where $\ell_1(\cdot)\=
 \gamma(\cdot)\, L_1$.
\end{rem}

Bott-type iteration formula for the Maslov-type index is a very powerful tool to study the stability problem, such as \cite{Bot56}, \cite{BTZ82} and
\cite{Lon99}. In \cite{HS09}, the authors generalized this iteration formula to the case with group action on the orbit. Here we just give the cyclic symmetry
case.

Let $Q$ be a fixed symplectic orthogonal matrix, $E$ be the function space
\begin{equation}
E=\{z\in W^{1,2}(\R/T\Z,\R^{2n})\mid z(t)=Qz(t+T)\}
 \end{equation}
 and $g$ be the
generator of $Z_m$, then the $Z_m$-group action is defined by $gz(t)=Sz(t+\frac{T}{m})$ for any $z\in E$, where $S$ is an orthogonal symplectic matrix
such that $JS=SJ$ and $S^m=Q$, we have
\begin{prop}\cite[Thm 1.1]{HS09}\label{Thm:I-F}
Let $z$ be a solution of the system \eqref{eq:ham-sys-1} and $\gamma(t)$ be the fundamental solution, then for the cyclic symmetry, we have
\begin{equation}\label{eq:I-F-C}
\iCLM(\Gr(\trasp{Q}),\Gr(\gamma(t)),t\in[0,T])=\sum_{i=1}^m\iCLM\big(\Gr(\exp(\dfrac{i}{m}2\pi\sqrt{-1})\trasp{S}),\Gr(\gamma(t)),t\in[0,T/m]\big).
\end{equation}
\end{prop}

First order Hamiltonian Systems encountered in the Applications come in general
from second order
Lagrangian Systems. It is well-known that in this case there is a direct
relation between the Maslov index and
the (classical) Morse index. For, let $L \in \mathscr C^2\big([0,T]\times
\R^{2n}, \R\big)$ be a {\em Lagrangian
function\/} and let $\mathbb S_L: W^{1,2}([0,T]; \R^n)\to \R$ be the {\em
Lagrangian action functional\/} defined as
\begin{equation}\label{eq:lagr-action}
 \mathbb S_L(x)\=\int_0^T L\big(t, x(t), \dot x(t)\big)\, dt.
\end{equation}
We assume that the function $L$ satisfying the Legendre convexity
condition:
\begin{equation}
 \Big \langle D_{vv}^2\,L(t,q,v)\, w, w \Big \rangle>0  \textrm{ for } t \in
[0,T], \
 w \in \R^n , \ (q,v) \in \R^n \times \R^n.
\end{equation}
Let $S$ be an orthogonal matrix, a solution of the
Euler-Lagrange Equation with the
boundary condition $\big(x(0), x(T)\big) \in \Gr(\trasp{S})$ is  a critical point
of $\mathbb S_L$ in the
space
\begin{equation}
 E_S\=\Set{x\in W^{1,2}([0,T], \R^n)| \big(x(0), x(T)\big) \in
\Gr(\trasp{S})}.
\end{equation}
For a critical point $x$ of $\mathbb S_L$, its Morse index is denoted by $m^-(x)$.

By using the Legendre transformation $p= D_v L(t, q,v)$ and setting $H(t,p,q)=
\langle p, v\rangle -
L(t,q, v)$ the Euler-Lagrange Equation can be converted into
the following Hamiltonian System
\begin{equation}\label{eq:ham-indotto}
 \dot z(t)= J\, \nabla H\big(t, z(t)\big)
\end{equation}
with the following Lagrangian  boundary condition
\[
 \big(z(0),z(T)\big)\in \Gr(\trasp{S}_d),
\]
where $S_d=\begin{bmatrix}S&0\\0&S\end{bmatrix}\in\Sp(2n)$.

Then we have the following useful {\em Morse-type Index
Theorem.\/}
\begin{prop}\cite[Thm 1.2]{HS09} \label{Thm:R-O-M-M}
Let $x$ be a critical point of
$\mathbb S_L$ and we assume that the Legendre convexity condition holds. Then
the
Morse index of $x$ is finite and the following holds:
\[
m^-(x) + \nu(S)= \iCLM\big(\Gr(\trasp{S}_d),\Gr(\Phi(t))\big)
\]\label{eq:M-M}
where $\nu(S)=\dim\ker(S-I_n)$ and $\Phi(t)$ is the fundamental solution of the corresponding Hamiltonian system.
\end{prop}
\begin{proof}
 For the proof of this result we refer the interested reader to \cite[Theorem
3.4]{HS09}. This concludes the proof.
\end{proof}
\begin{rem}\label{Rem:C-C}
The equation in Proposition \ref{Thm:R-O-M-M} also holds in the complex case if we assume $S$ is unitary on $\C^n$. For further details we refer the interested reader
to \cite[Remark 3.6]{HS09}.
\end{rem}

We conclude this Section with the following proposition from \cite[Cor.2.1]{LZ00b}.
\begin{prop}\label{prop:L-C}
For any symplectic path $\Phi$ starting from $I_{2n}$, we have
\begin{equation}
\iota_1(\Phi)+n=\iCLM\big(\Delta,\Gr(\Phi(t))\big),
\end{equation}
and
\begin{equation}
\iota_\omega(\Phi)=\iCLM\big(\Gr(\omega),\Gr(\Phi(t))\big),\qquad \forall\,  \omega\in\U\setminus\{1\},
\end{equation}
where $\Delta=\Gr(I_{2n})$ and $\Gr(\omega)=\Gr(\omega I_{2n})$.
\end{prop}


\vspace{1cm}
	\noindent
	\textsc{Prof. Xijun Hu}\\
	Department of Mathematics\\
	Shandong University\\
	Jinan, Shandong, 250100 \\
	The People's Republic of China \\
	China\\
	E-mail: \email{xjhu@sdu.edu.cn}

\vspace{1cm}
\noindent
\textsc{Prof. Alessandro Portaluri}\\
DISAFA\\
Università degli Studi di Torino\\
Largo Paolo Braccini 2 \\
10095, Grugliasco, Torino\\
Italy\\
Website: \url{aportaluri.wordpress.com}\\
E-mail: \email{alessandro.portaluri@unito.it}

\vspace{1cm}
\noindent
\textsc{Dr. Ran Yang}\\
School of Science\\
East China  University of Technology\\
Nanchang, Jiangxi, 330013\\
The People's Republic of China \\
China\\
E-mail: \email{yangran201311260@mail.sdu.edu.cn}

\vspace{1cm}
\noindent
COMPAT-ERC Website: \url{https://compaterc.wordpress.com/}\\
COMPAT-ERC Webmaster \& Webdesigner: Arch.  Annalisa Piccolo

\end{document}